\documentclass[12pt]{amsart}
\usepackage{amsfonts,amsmath,amssymb,amsthm,geometry,url,hyperref, graphicx}
\usepackage{amsmath, amssymb, paralist,tabularx,supertabular,amsmath, verbatim, doc, amsthm, amssymb, mathrsfs, manfnt,palatino, color,amsxtra, multirow}
\usepackage{fullpage,multicol}
\usepackage{amsfonts,amsmath,amssymb,amsthm,geometry,url,hyperref}
\usepackage{amsmath, amssymb, paralist,tabularx,supertabular,amsmath, verbatim, doc, amsthm, amssymb, mathrsfs, manfnt,palatino, color,amsxtra, multirow,xcolor,colortbl}
\let\normalint\int % PS
\def\int{\displaystyle\normalint} %PS

\numberwithin{equation}{section}
\theoremstyle{remark}

\newtheorem{example}[equation]{Example}

\newtheorem{rmk}[equation]{Remark}

\theoremstyle{plain}
\newtheorem{thm}[equation]{Theorem}
\newtheorem{theorem}[equation]{Theorem}
\newtheorem{lemma}[equation]{Lemma}
\newtheorem{proposition}[equation]{Proposition}

\newtheorem*{thmmain}{Main Theorem}

\newtheorem{lem}[equation]{Lemma}

\newtheorem{corollary}[equation]{Corollary}

\DeclareMathOperator{\Aut}{Aut}

\DeclareMathOperator{\coarea}{coarea}

\DeclareMathOperator{\covol}{covol}

\DeclareMathOperator{\Gal}{Gal}

\DeclareMathOperator{\impart}{Im}

\DeclareMathOperator{\Isom}{Isom}

\DeclareMathOperator{\lcm}{lcm}

\DeclareMathOperator{\M}{M}

\DeclareMathOperator{\N}{N}

\DeclareMathOperator{\nrd}{nrd}

\DeclareMathOperator{\PGL}{PGL}

\DeclareMathOperator{\PSL}{PSL}
\DeclareMathOperator{\Ram}{Ram}

\DeclareMathOperator{\Reg}{Reg}

\DeclareMathOperator{\SL}{SL}

\DeclareMathOperator{\trd}{trd}

\DeclareMathOperator{\vol}{vol}
\newcommand{\C}{\mathbb C}

\newcommand{\HH}{\mathbb H}
\newcommand{\PP}{\mathbb P}
\newcommand{\Q}{\mathbb Q}
\newcommand{\R}{\mathbb R}
\newcommand{\Z}{\mathbb Z}

\newcommand{\bfP}{\mathbb P}

\newcommand{\frakd}{\mathfrak{d}}

\newcommand{\frakl}{\mathfrak{l}}

\newcommand{\frakp}{\mathfrak{p}}
\newcommand{\frakq}{\mathfrak{q}}

\newcommand{\frakD}{\mathfrak{D}}

\newcommand{\frakN}{\mathfrak{N}}

\newcommand{\frakP}{\mathfrak{P}}

\newcommand{\fraks}{\nu}

\newcommand{\scrC}{\mathscr{C}}

\newcommand{\scrS}{\mathscr{S}}

\newcommand{\calE}{\mathcal{E}}

\newcommand{\calH}{\mathcal{H}}

\newcommand{\calO}{\mathcal{O}}

\newcommand{\eps}{\epsilon}

\newcommand{\Htwo}{\mathbf{H}^2}
\newcommand{\Hthree}{\mathbf{H}^3}

\newcommand{\Dzambic}{D$\check{\text{z}}$ambi\'c}
\newcommand{\quat}[2]{\displaystyle{\biggl(\frac{#1}{#2}\biggr)}}

\def\mc{\mathcal}

\newcommand{\defi}[1]{{{\fontfamily{lmss}\selectfont \textit{#1}}}}
\begin{document}
\title{Commensurability classes of fake quadrics}
\author{Benjamin Linowitz}
\address{Department of Mathematics, Oberlin College, 10 North Professor Street, Oberlin, OH 44074, USA}
\email{benjamin.linowitz@oberlin.edu}
\author{Matthew Stover}
\address{Department of Mathematics, Temple University, 1805 N.\ Broad Street, Philadelphia, PA 19122, USA}
\email{mstover@temple.edu}
\author{John Voight}
\address{Department of Mathematics, Dartmouth College, 6188 Kemeny Hall, Hanover, NH 03755, USA}
\email{jvoight@gmail.com}

%%%%%%%%%%%%%%%%%%%%%%%%%

\maketitle

\begin{abstract}
A fake quadric is a smooth projective surface that has the same rational cohomology as a smooth quadric surface but is not biholomorphic to one. We provide an explicit classification of all irreducible fake quadrics according to the commensurability class of their fundamental group. To accomplish this task, we develop a number of new techniques that explicitly bound the arithmetic invariants of a fake quadric and more generally of an arithmetic manifold of bounded volume arising from a form of $\SL_2$ over a number field.
\end{abstract}

%%%%%%%%%%%%%%%%%%%%%%%%%
\subsection*{Introduction}
%%%%%%%%%%%%%%%%%%%%%%%%%

Mumford asked in the early 1980s whether surfaces with geometric genus $p_g = 0$ can be enumerated by computer. A major step toward answering this in the affirmative was the classification of fake projective planes, smooth projective surfaces with the same rational cohomology as $\PP_\C^2$ but not biholomorphic to it, by Prasad--Yeung \cite{Prasad--Yeung} and Cartwright--Steger \cite{Cartwright--Steger}. Another important family of surfaces with $p_g = 0$ are \defi{fake quadrics}, which are smooth projective surfaces with the same rational cohomology as, but not biholomorphic to, the quadric surface $\PP_\C^1 \times \PP_\C^1$.

The \defi{reducible} fake quadrics, those finitely covered by a direct product of smooth projective curves, were classified by Bauer--Catanese--Grunewald \cite{Bauer--Catanese--Grunewald}, and these fall into $18$ families \cite{Bauer--Catanese--Grunewald2}. A fake quadric uniformized by the product $\calH = \Htwo \times \Htwo$ of two hyperbolic planes that is not reducible is called \defi{irreducible}; irreducible fake quadrics are necessarily \defi{arithmetic} (see \S \ref{sec:Arithmetic}) and thus are finite covers of \defi{quaternionic Shimura surfaces}. Examples of irreducible fake quadrics are known \cite{Shavel, Takeuchi, Dzambic}, and the contribution of this paper, analogous to the work of Prasad--Yeung \cite{Prasad--Yeung}, is to classify the maximal irreducible lattices in the holomorphic isometry group $\Isom^h(\calH)$ of $\calH$ that contain the fundamental group of a fake quadric.

The broad strategy of our classification is to enumerate maximal arithmetic lattices of bounded covolume, which has been employed to solve a number of similar problems. However, the tools used in other contexts are thoroughly insufficient for enumerating fake quadrics. This paper introduces several new techniques that bring this problem into practical range. Our methods are relevant for studying arbitrary forms of $\SL_2$ over number fields, where volume estimates for arithmetic lattices are often most subtle, and we present them in this generality.

Another important feature of this paper that does not appear in previous work on arithmetic fake quadrics and their generalizations is that we study \emph{all} maximal lattices that can contain the fundamental group of a fake quadric. Previous papers only consider lattices in $\PGL_2^+(\R) \times \PGL_2^+(\R)$, which has index $2$ in $\Isom^h(\calH)$; the quotient is generated by the \defi{swap map}, the holomorphic isometry that exchanges the two factors of $\calH$. In particular, there are torsion-free lattices $\Gamma$ in $\Isom^h(\calH)$ not contained in $\PGL_2^+(\R) \times \PGL_2^+(\R)$. We define the \defi{stable subgroup} $\Gamma_{st}$ of a lattice $\Gamma \leq \Isom^h(\calH)$ to be the subgroup of elements that preserve the factors of $\calH$. As described below, the proof of our main results then follows from classifying the stable lattices of covolume equal to twice the volume of a fake quadric. Due to the nature of the analytic bounds that one employs in this situation, the extra factor of $2$ turns out to be quite significant. It is of both geometric and arithmetic interest when the stable subgroup $\Gamma_{st}$ of a lattice $\Gamma \leq \Isom(\calH)$ is a proper subgroup, and we give a complete characterization in Theorem \ref{thm:galoisdescswap}. For example, we show that $\Gamma_{st}=\Gamma$ when $\Aut(k)$ is trivial, where $k$ is the trace field of $\Gamma$. 
% It would be instructive to understand further arithmetic consequences of Theorem \ref{thm:galoisdescswap}.

%Therefore, the quotients $X$ we obtain are (finite covers of) 

%If $X$ is a fake quadric, then
%\begin{equation}\label{eq:Chern}
%c_1(X)^2 = 2 c_2(X) = 8,
%\end{equation}
%where $c_j$ denotes the $j$th Chern number (we discuss other numerical invariants in \S \ref{sec:background}). It follows from Hirzebruch's proportionality principle that one way to construct surfaces satisfying \eqref{eq:Chern} is to take a quotient of the product $\calH = \Htwo \times \Htwo$ of hyperbolic planes by a discrete, torsion-free group $\Gamma \leq \Isom^+(\calH)$ of orientation-preserving isometries such that $\Gamma$ has finite abelianization and the quotient $X = \Gamma \backslash \calH$ has (topological) Euler characteristic $4$. It is unknown if all fake quadrics are uniformized by $\calH$.

%Mostow--Siu rigidity \cite{Siu} implies that an irreducible fake quadric is a pair of isolated points in the moduli space of fake quadrics.

%%%%%%%%%%%%%%%%%%%%%%%%%
%\subsection*{Main result}
%%%%%%%%%%%%%%%%%%%%%%%%%

The main result of this paper is the following theorem.

%%%%%%%%%%%%%%%%%%%%%%%%%
\begin{thmmain}\label{thm:Classification}
Let $X=\Gamma \backslash \calH$ be an irreducible fake quadric. Then $\Gamma$ is a subgroup of one of the maximal arithmetic lattices in $\Isom^h(\calH)$ enumerated in the Appendix.
\end{thmmain}
%%%%%%%%%%%%%%%%%%%%%%%%%

Briefly, the Main Theorem is proven as follows. Equip $\Htwo$ with the usual metric of constant curvature $-1$ and $\calH$ with the product metric. By the Chern--Gauss--Bonnet theorem, any fake quadric $X$ uniformized by $\calH$ has volume $\vol(X)=16 \pi^2$ in the induced metric. Conversely, every irreducible torsion-free lattice in $\Isom^h(\calH)$ of covolume $16 \pi^2$ acting by holomorphic isometries necessarily determines a fake quadric. We then classify all maximal irreducible lattices in $\Isom^h(\calH)$ that can possibly contain a torsion-free subgroup of covolume $16 \pi^2$. The entries in the Appendix are indexed by commensurability class, so the torsion-free subgroups of the lattices in the Appendix with the appropriate index describe all irreducible fake quadrics. % In a sequel, analogous to the work of Cartwright--Steger \cite{Cartwright--Steger}, we will explicitly present the fundamental groups of fake quadrics that occur in each such commensurability class and study the geometry and topology of the quotients in more detail.

Previous papers studying covolume for forms of $\SL_2$ rely heavily on the classic paper of Chinburg--Friedman \cite{chinburg-smallestorbifold} on the smallest volume arithmetic hyperbolic $3$-orbifold. Their methods are based upon finding an upper bound for the degree over $\Q$ of the field of definition for the arithmetic group, deducing upper bounds for the order of a certain class group, and then reducing the classification to a finite calculation. In this paper, crucially we present a two-pronged variant of the Chinburg--Friedman method (Theorem \ref{theorem:lemma43generalization}) where one bound is preferable for fields of large degree, the other preferable in small degree but potentially large class number (the main technical obstruction found in previous work). This is still not sufficient to complete the classification, and our second method is the \emph{Fuchsian gambit} (Proposition \ref{prop:quatgambit}), which exploits the geometry of a closely related Shimura curve and the classical expression for its signature to bound the order of a certain class group.

%%%%%%%%%%%%%%%%%%%%%%%%%
%\begin{thmb}\label{thm:Classification2}
%There are at most {\color{red}???} maximal irreducible arithmetic lattices in $\Isom^+(\calH)$ that contain a torsion free sublattice $\Gamma$ with $X=\Gamma \backslash \calH$ having $\vol(X)=16\pi^2$. 
%\end{thmb}
%%%%%%%%%%%%%%%%%%%%%%%%%

%Each of the potential 1418 maximal lattices is listed in the Appendix.

%%%%%%%%%%%%%%%%%%%%%%%%%
\subsection*{Relation with previous work}
%%%%%%%%%%%%%%%%%%%%%%%%%

%The basic ingredients in the proof of Theorem \ref{thm:Classification2} are Borel's volume formula \cite{borel-commensurability} and Odlyzko's discriminant bounds \cite{Odlyzko-bounds}, which together reduce the task to a finite computation.
The basic strategy of this paper was used for classification purposes in several other important contexts. Most notably, Chinburg--Friedman \cite{chinburg-smallestorbifold} determined the arithmetic hyperbolic $3$-orbifold with smallest volume, Chinburg--Friedman--Jones--Reid \cite{CFJR} found the arithmetic hyperbolic $3$-manifold of minimal volume, and Prasad--Yeung \cite{Prasad--Yeung} and Cartwright--Steger \cite{Cartwright--Steger} classified fake projective planes with this circle of ideas. Prasad--Yeung also considered higher-dimensional fake spaces \cite{Prasad--Yeung2, Prasad--Yeung3}. Belolipetsky \cite{Belolipetsky} and Belolipetsky--Emery \cite{BE} also studied higher-dimensional hyperbolic orbifolds and manifolds; see also the recent survey by Belolipetsky \cite{misha:ICM} and the references therein.

%Our paper follows the same outline and owes a great deal to the circle of ideas in these papers, but unfortunately the tools themselves do not work ``out of the box''. The lattices we consider here have much larger covolume, and one essential contribution of this paper is to develop stronger tools to bring this calculation to within practical range. Briefly, previous approaches are based upon finding an upper bound for the degree over $\Q$ of the field of definition for the arithmetic group, deducing upper bounds for the order of a certain class group, then squeezing out a classification. We prove a variant of the Chinburg--Friedman bound to extract stronger class number bounds in small degree. Theorem \ref{theorem:lemma43generalization} gives two lower bounds for covolumes of arithmetic lattices; . 

Independently, \Dzambic\ \cite{Dzambic} obtained bounds on the degree of a totally real field $k$ giving rise to an irreducible fake quadric under the additional hypothesis that the fundamental group is stable in the sense described above. Unfortunately, even when restricting to stable groups, our lists do not agree (Remark \ref{rmk:dzambicerror}). While these two papers share the same general outline, the scope and strategy of the calculations are quite different.

%%%%%%%%%%%%%%%%%%%%%%%%%
\subsection*{Contents}
%%%%%%%%%%%%%%%%%%%%%%%%%

In \S \ref{sec:background}, we describe some basic results on the geometry of fake quadrics and define the basic number theoretic objects used throughout this paper. In \S \ref{sec:Arithmetic}, we get into the finer details of arithmetic lattices arising from quaternion algebras, characterize stable subgroups, and discuss torsion in arithmetic lattices. The technical heart of the paper is \S \ref{ssec:chinburgfriedmanlemma43}, where we generalize volume bounds of Chinburg and Friedman \cite{chinburg-smallestorbifold} for irreducible arithmetic lattices acting on arbitrary products of hyperbolic $2$- and $3$-spaces and introduce the Fuchsian gambit. In \S \ref{sec:enumeration}, we use these bounds to put strong restrictions on the number fields associated with irreducible fake quadrics, then prove some additional constraints and describe how one can enumerate the commensurability classes of irreducible fake quadrics using \textsc{Magma}. The Appendix presents the results of these calculations.

The authors would like to thank Carl Pomerance and Alan Reid for advice and the referees for their constructive feedback and comments on the code. The first author was partially supported by an NSF RTG grant DMS-1045119 and an NSF Mathematical Sciences Postdoctoral Fellowship. The second author was supported by NSF grant DMS-1361000, and acknowledges support from U.S. National Science Foundation grants DMS 1107452, 1107263, 1107367 "RNMS: GEometric structures And Representation varieties" (the GEAR Network). The third author was supported by an NSF CAREER Award (DMS-1151047).

%%%%%%%%%%%%%%%%%%%%%%%%%
\section{Background} \label{sec:background}
%%%%%%%%%%%%%%%%%%%%%%%%%

%%%%%%%%%%%%%%%%%%%%%%%%%
\subsection*{Geometry of fake quadrics}
%%%%%%%%%%%%%%%%%%%%%%%%%

This section describes the basic results on fake quadrics we use. A good reference is the book of Griffiths--Harris \cite{Griffiths--Harris}.

Let $X$ be a smooth projective surface. The real cohomology of $X$ then has a Hodge decomposition
\[
H^k(X, \R) \cong \bigoplus_{i + j = k} H^{i, j}(X).
\]
If $h^{i,j}(X)=\dim_\R H^{i,j}(X)$, then, since $X$ is projective and hence K\"ahler, $h^{i, j}(X) = h^{j, i}(X)$. The \defi{holomorphic Euler characteristic} is
\[
\chi(X) = h^{0,0}(X) - h^{0,1}(X) + h^{0,2}(X).
\]
If $c_j(X)$ is the $j$th Chern number of $X$, then $c_2(X)$ is the topological Euler characteristic and $c_1^2(X)$ is the self-intersection of the canonical bundle of $X$. These quantities are related by Noether's formula:
\begin{equation}\label{eq:Noether}
\chi(X) = \frac{c_1^2(X) + c_2(X)}{12}.
\end{equation}

A \defi{fake quadric} is a smooth projective surface $X$ with the same rational cohomology as $Q = \PP_\C^1 \times \PP_\C^1$ but not isomorphic to it. If $X$ is a fake quadric, then $X$ has geometric genus $p_g(X) = h^{0,2}(X) = 0$ and $\chi(X) = 1$.

If $\Htwo = \{ z = x + y i \in \C : y = \impart(z) > 0 \}$ is the upper half-plane model for hyperbolic $2$-space, equip $\Htwo$ with the metric $ds^2 = (dx^2 + dy^2) / y^2$, so that $\Htwo$ has constant negative curvature $-1$. Let $\calH=\Htwo \times \Htwo$ and $G = \Isom^h(\calH)$ be the group of holomorphic isometries of $\calH$ (see \S \ref{sec:Arithmetic} for more on this group) and $\Gamma \leq G$ be a torsion-free, discrete subgroup such that the quotient $X = \Gamma \backslash \calH$ is compact. Then $X$ is a smooth projective surface. Hirzebruch's proportionality principle \cite{Hirzebruch} says that the ratio 
\[
\frac{c_1^2(X)}{c_2(X)} = \frac{c_1^2(\calH\spcheck)}{c_2(\calH\spcheck)} = 2,
\] 
since $\calH\spcheck = \bfP_\C^1 \times \bfP_\C^1$ is the compact complex dual to $\calH$. It follows that $c_1^2(X) = 2 c_2(X)$, and hence, by \eqref{eq:Noether}, $c_2(X) = 4 \chi(X)$.

A surface $X = \Gamma \backslash \calH$ as above is \defi{irreducible} if it is not finitely covered by a direct product of (complex) curves. If $X$ is irreducible, then a theorem of Matsushima--Shimura \cite{Matsushima--Shimura} implies that $H^1(X, \Q) = \{0\}$, so $h^{0,1} = h^{1,0} = 0$. Applying the Chern--Gauss--Bonnet theorem, we find that
\[
\vol(X) = (2 \pi)^2 c_2(X) = (16 \pi^2) \chi(X).
\]
Putting these together, we have shown that if $X$ is an irreducible fake quadric, then:
\[
\chi(X) = 1 \quad c_2(X) = 4 \quad c_1(X)^2 = 8 \quad \vol(X) = 16\pi^2
\]
Furthermore, the the latter equality plus irreducibility is sufficient for the converse.

%%%%%%%%%%%%%%%%%%%%%%%%%
\begin{lem}\label{lem:Numerical}
Suppose $\Gamma \leq \Isom^h(\calH)$ is a torsion-free lattice and $X = \Gamma \backslash \calH$ is compact and irreducible. Then $X$ is a fake quadric if and only if $\vol(X) = 16 \pi^2$.
\end{lem}
%%%%%%%%%%%%%%%%%%%%%%%%%

Therefore, our task of classifying fake quadrics becomes the problem of enumerating torsion-free lattices $\Gamma \leq \Isom^h(\calH)$ such that $\vol(\Gamma \backslash \calH) = 16 \pi^2$.

%%%%%%%%%%%%%%%%%%%%%%%%%
\subsection*{Number fields and quaternion algebras}\label{sec:NumThy}
%%%%%%%%%%%%%%%%%%%%%%%%%

Basic references for quaternion algebras are Maclachlan--Reid \cite{Maclachlan--Reid} and Vign\'eras \cite{vigneras1980}.

Let $k$ be a number field with degree $n = [k : \Q]$, signature $(r_1,r_2)$, absolute discriminant $d$, and root discriminant $\delta=d^{1/n}$. Also, let $\calO_k$ be the ring of integers of $k$ and $\zeta_k(s)$ its Dedekind zeta function. We let $w$ be the number of roots of unity contained in $k$, $h$ be the class number, and $\Reg$ be the regulator of $k$.

If $\calO_k^\times$ is the unit group of $\calO_k$ and $\calO_{k, +}^\times \leq \calO_k^\times$ is the subgroup of \defi{totally positive units}, consisting of units that are positive under every embedding of $k$ into $\R$, then
\[
\calO_{k,+}^\times / \calO_k^{*2} \cong (\Z/2\Z)^{m}
\]
for some $m \in \Z_{\geq 0}$. When $k$ is not totally complex, we have $-1 \notin \calO_{k, +}^\times$ and hence $m \ge 1$. Since $\calO_k^{*2}$ has index $2^{n}$ in $\calO_k^\times$, it follows that
\[
[\calO_k^\times : \calO_{k, +}^{*2}] = 2^{n - m} \ge 2.
\]

A \defi{quaternion algebra} $B$ over $k$ is a central simple algebra of dimension $4$ over $k$, or equivalently, a $k$-algebra with generators $i,j \in B$ such that
\[
i^2 = a, \quad j^2 =b, \quad j i = -i j
\]
with $a, b \in k^\times$. An algebra described in by $i, j$ as above is denoted $B=\quat{a,b}{k}$. A quaternion algebra $B$ over $k$ has a unique (anti-)involution $\overline{\phantom{x}}:B \to B$ such that the \defi{reduced trace} $\trd(\alpha)=\alpha+\overline{\alpha}$ and the \defi{reduced norm} $\nrd(\alpha)=\alpha\overline{\alpha}$ belong to $k$ for all $\alpha \in B$.

Let $v$ be a place of $k$ and $k_v$ denote the completion of $k$ at $v$. Then $B$ is \defi{split} at $v$ if $B_v=B \otimes_k {k_v} \cong \M_2(k_v)$, and otherwise $B$ is \defi{ramified} at $v$ (in which case $B_v$ is a division ring). The set $\Ram(B)$ of ramified places of $B$ is a finite set of noncomplex places of $k$ having even cardinality. Further, the set $\Ram(B)$ uniquely characterizes $B$ up to $k$-algebra isomorphism. Let $\Ram_\infty(B)$ (resp.\ $\Ram_f(B)$) be the infinite (resp.\ finite) places of $k$ in $\Ram(B)$.

An \defi{order} $\calO \subset B$ is a subring of $B$ (with $1 \in \calO$) that is finitely generated as an $\calO_k$-submodule and that contains a basis for $B$ over $k$. An order is \defi{maximal} if it is not properly contained in another order. Every order is contained in a maximal order, and there are finitely many maximal orders in $B$ up to conjugation by $B^\times$.
%%%%%%%%%%%%%%%%%%%%%%%%%
\section{Arithmetic lattices and quaternion algebras}\label{sec:Arithmetic}
%%%%%%%%%%%%%%%%%%%%%%%%%

In this section, we describe the construction of lattices arising from quaternion algebras. We retain the notation from \S \ref{sec:background}. In particular, $k$ is a number field and $B$ is a quaternion algebra over $k$.

%%%%%%%%%%%%%%%%%%%%%%%%%
\subsection*{Lattices potentially yielding fake quadrics}
%%%%%%%%%%%%%%%%%%%%%%%%%

Suppose that $B$ splits at precisely $a$ real places and let $b=r_2$. Then
\[
B_\R = B \otimes_\Q \R \cong \mathrm{M}_2(\R)^a \oplus \HH^{r_1 - a} \oplus \M_2(\C)^{b}
\]
where $\HH$ is the real division algebra of Hamiltonian quaternions. Let $\iota:B_\R \hookrightarrow \M_2(\R)^a$ be the embedding obtained by composing the above map with the projection onto the factor $\M_2(\R)^a \times \M_2(\C)^{b}$.

Given a maximal order $\calO$ of $B$, let $\Gamma_{\calO}$ be the image of $\calO^\times$ in $G_{a,b} = \PGL_2(\R)^a \times \PGL_2(\C)^b$ under $\iota$. Then $\Gamma_{\calO}$ is a lattice by Borel and Harish-Chandra \cite{borel-harish-chandra}, and the image of $\Gamma_{\calO}$ under projection onto any proper factor of $G_{a,b}$ is topologically dense.

In this paper we are interested in subgroups acting discretely on $\calH = \Htwo \times \Htwo$. Therefore, to construct fake quadrics we take $a = 2$. Accordingly, we suppose from now on that $k$ is totally real and $B$ is ramified at all but exactly two real places of $k$. In particular, $n = [k:\Q] \geq 2$.

A (discrete) subgroup $\Gamma \leq \Isom(\calH)$ commensurable with $\Gamma_{\calO}$ for some maximal order $\calO \subset B$ (in the narrow sense, with the intersection having finite index in each group) is called an \defi{arithmetic lattice} in $\Isom(\calH)$. An arithmetic lattice $\Gamma$ is \defi{reducible} if it contains a finite index subgroup $\Gamma^\prime \leq \Gamma$ such that $\Gamma^\prime \cong \Gamma_1 \times \Gamma_2$ is isomorphic as a group to a product of groups with each $\Gamma_i \leq \Isom(\Htwo)$ a Fuchsian group. This paper is concerned only with irreducible arithmetic lattices, and by the arithmeticity theorem of Margulis and the classification of $k$-forms of $\PGL_2$, all irreducible arithmetic lattices are indeed obtained from the above construction.

Let $\Isom^h(\calH)$ be the subgroup of $\Isom(\calH)$ consisting of holomorphic isometries, $\Gamma \leq \Isom(\calH)$ be an arithmetic lattice, and
\[
X = \Gamma \backslash \calH.
\]
Then $X$ has the structure of a complex orbifold of dimension $2$. By the Godement compactness criterion, $X$ is compact (equivalently, $\Gamma$ is \defi{uniform}) if and only if $B$ is a division algebra over $k$ (equivalently, $\Ram(B) \neq \emptyset$). We summarize the above discussion with the following lemma.

%%%%%%%%%%%%%%%%%%%%%%%%%
\begin{lem}
Commensurability classes of irreducible, cocompact arithmetic lattices in $\Isom(\calH)$ are determined by precisely those forms of $\PGL_2$ defined by quaternion division algebras $B$ over a totally real field $k$ such that $\Ram(B)$ contains all but exactly two real places of $F$. 
\end{lem}
%%%%%%%%%%%%%%%%%%%%%%%%%

%%%%%%%%%%%%%%%%%%%%%%%%%
\subsection*{Stable subgroups}
%%%%%%%%%%%%%%%%%%%%%%%%%

Let $G=\Isom^+(\calH) \leq \Isom(\calH)$ denote the orientation-preserving isometry group of $\calH$. Then $\Isom^+(\calH)$ contains the \defi{swap map}
\begin{equation} \label{eqn:swapmap}
\begin{aligned}
\sigma : \calH &\to \calH \\
(z_1, z_2) &\mapsto (z_2, z_1).
\end{aligned}
\end{equation}
By the de Rham decomposition theorem \cite{deRham}, $G$ is generated by $\sigma$ and the \defi{stable subgroup}
\[
G^* = G \cap (\Isom(\Htwo) \times \Isom(\Htwo)).
\]
The group $\Isom(\Htwo) \cong \PGL_2(\R)$ is generated by $\PGL_2^+(\R)$ acting by linear fractional transformations and the orientation-reversing isometry $\rho(z)=-1/\overline{z}$ of negative determinant. Therefore
\[
G^* = \{ \gamma = (\gamma_1, \gamma_2) \in \PGL_2(\R) \times \PGL_2(\R) : \det(\gamma_1) \det(\gamma_2) > 0 \}.
\]
When $\gamma \in \PGL_2(\R) \times \PGL_2(\R)$ has $\det(\gamma_i) < 0$ for $i = 1, 2$, then the corresponding orientation-preserving isometry acts on each factor by an orientation-reversing isometry. Notice that such an isometry does not act holomorphically on $\calH$.

Finally, inside $G^*$ is the subgroup $G^+ \cong \PGL_2^+(\R) \times \PGL_2^+(\R)$ of index $2$. Therefore, we have inclusions
\[
\Isom^+(\calH) = G \overset{2}{>} G^* \overset{2}{>} G^+ \cong \PGL_2^+(\R) \times \PGL_2^+(\R).
\]
A discrete subgroup $\Gamma \leq G$ is called \defi{stable} if $\Gamma \leq G^*$. A subgroup $\Gamma$ of $G$ acts by holomorphic isometries if and only if it is contained in
\[
\Isom^h(\calH) = \left\langle \sigma, \PGL_2^+(\R) \times \PGL_2^+(\R) \right\rangle,
\]
and $\Isom^h(\calH) \cap G^* = G^+$.

Critical for enumerating all lattices in $\Isom^h(\calH)$ of covolume $16 \pi^2$ is the following, which characterizes those lattices not contained in $G^*$. See also Gr\r{a}nath \cite[Section 4.3]{granath2002quaternionic} and \Dzambic--Roulleau \cite[Proposition 6]{dzambicroulleau:inv}.

%%%%%%%%%%%%%%%%%%%%%%%%%
\begin{thm} \label{thm:galoisdescswap}
Let $\Gamma_{st} \leq G^*$ be an arithmetic lattice associated with a quaternion algebra $B$ over $k$, and let 
$\iota : B \hookrightarrow \M_2(\R)^2$ be an algebra embedding. Then
\[
N_G(\Gamma_{st}) \cap G^* \leq \iota(B^\times/k^\times).
\]
Moreover, there exists a lattice $\Gamma \leq G$ properly containing $\Gamma_{st}$ such that $\Gamma \not\leq G^*$ if and only if there exists $\tau \in \Aut(k / \Q)$ such that
\begin{enumerate}

\item[\textup{(i)}] $\tau(\Ram(B)) = \Ram(B)$ and $\tau$ interchanges the two real places of $k$ that split $B$ and

\item[\textup{(ii)}] $\tau(\Gamma_{st})$ is conjugate to $\Gamma_{st}$ under the induced action of $\iota(B^\times)$.

\end{enumerate}
\end{thm}
%%%%%%%%%%%%%%%%%%%%%%%%%

%%%%%%%%%%%%%%%%%%%%%%%%%
\begin{proof}
Suppose that $\alpha \in G^*$ normalizes $\Gamma_{st}$. One way to prove the first statement is to note that $\alpha$ lies in the commensurator of $\Gamma_{st} \cap G^*$ in $G^*$ then apply a result of Borel \cite[Thm.\ 3(b)]{BorelMaximality} which implies that the commensurator in $G^*$ of any arithmetic subgroup of $G^*$ commensurable with $\Gamma_{st} \cap G^*$ equals $\iota(B^\times)$.

We give another proof, which is useful for the proof of the second part. Consider
\[
\Gamma_{st}^1 = \Gamma_{st} \cap \iota(B^1) \le \PSL_2(\R) \times \PSL_2(\R),
\]
where $B^1 \leq B^\times$ is the subgroup of units with reduced norm $1$. Then
\[
\iota(B^1) \cong B^1 / \{\pm 1\} \hookrightarrow G.
\]
Let $\Delta_{st}^1$ be the lift of $\Gamma_{st}^1$ to $B^1 \le \SL_2(\R) \times \SL_2(\R)$. We have a central exact sequence
\[
1 \to \Z / 2 \Z \to \Delta_{st}^1 \to \Gamma_{st}^1 \to 1.
\]
We claim that, replacing $\Gamma_{st}$ with a subgroup of finite index, we can assume that $\Delta_{st}^1 \cong \Gamma_{st}^1 \times \Z / 2 \Z$. By Selberg's lemma, $\Delta_{st}^1$ has a torsion-free finite-index subgroup $\Lambda_{st}^1$, which then projects isomorphically onto its image in $\Gamma_{st}^1$. Replacing $\Gamma_{st}^1$ with $\Lambda_{st}^1$, we see that the above central exact sequence admits a section that is a homomorphism, and the claim follows. Replacing $\Gamma_{st}^1$ with a further subgroup of finite index, we can assume that it is normalized by our element $\alpha \in G^*$.

The map $\iota:B \hookrightarrow \M_2(\R)^2$ is an embedding, so conjugation by $\alpha$ defines a $k$-algebra automorphism $\phi_\alpha \in \Aut_k(B)$. By the Skolem--Noether theorem, there exists $\beta \in B^\times$ such that $\phi_\alpha(x) = \beta x \beta^{-1}$ for all $x \in B$. Thus $\iota(\beta) \alpha^{-1}$ centralizes $\Delta_{st}^1$, and $G^*$ has no center, so $\alpha \in \iota(B^\times)$ as claimed. 

For the second statement, choose $\gamma \in \Gamma$ not in $G^*$. Recall that $\gamma$ is of the form $g \sigma$ for some $g \in G^*$ and $\sigma$ the swap map. Passing to a group $\Delta_{st}^1$ as above, we see that $\gamma$ defines an automorphism $\phi_\gamma$ of $B$ that is $\Q$-linear but not $k$-linear. More specifically, $\phi_\gamma$ is the composition of an element $\tau$ of $\Aut(k / \Q)$ that swaps the two real places of $k$ that split $B$ and a $k$-automorphism of $B$. Again, Skolem--Noether implies that there is some $\beta \in B^\times$ such that $\phi_\gamma(x) = \beta \tau(x) \beta^{-1}$ for all $x \in B$. Since the action of $\tau$ on $B^\times$ projects to the swap map $\sigma$, we see in the same way as above that $\tau(\Gamma_{st})$ is conjugate to $\Gamma_{st}$ by an element of $\iota(B^\times)$. Furthermore, this implies that $B \otimes_k k_{\tau(v)}$ is isomorphic to $B \otimes_k k_v$ for every place $v$ of $k$, so $\tau(\Ram(B)) = \Ram(B)$. This proves that (i) and (ii) hold.

Conversely, if conditions (i) and (ii) hold with $\tau(\Gamma_{st}) = \alpha \Gamma_{st} \alpha^{-1}$, define $\rho = \alpha^{-1} \sigma$. By (i), conjugation by $\sigma$ is the same as the action of $\tau$, and it follows that $\rho$ normalizes $\Gamma_{st}$. In particular, the subgroup of $G$ generated by $\rho$ and $\Gamma$ is a lattice, since the normalizer of $\Gamma$ in $G$ is a lattice, and hence is a finite extension of $\Gamma_{st}$.
\end{proof}
%%%%%%%%%%%%%%%%%%%%%%%%%

%The extension of $\Gamma_{st}$ given by (i) and (ii) might be torsion-free and so give rise in our context to a fake quadric. Many authors work only with subgroups of isometries that preserve each factor (though conventions differ as to whether or not to restrict to orientation-preserving subgroups). In order to give a complete classification of fake quadrics, we must work in this largest generality.

%%%%%%%%%%%%%%%%%%%%%%%%%
\begin{example}
In Theorem \ref{thm:galoisdescswap}, if $\tau(\Gamma_{st})=\Gamma_{st}$, then we obtain $\Gamma$ by adjoining the swapping map $\sigma$ itself. 
\end{example}

\begin{rmk}
It is possible to find a group where $\tau(\Gamma_{st}) \neq \Gamma_{st}$ as follows. Assume for simplicity that $k$ is Galois over $\Q$. Take an Eichler order $\calO$ of square-free level $\frakN$ such that $\tau(\calO)$ is isomorphic but not equal to $\calO$. For example, take an Eichler order $\calO$ of level $\frakN = \tau(\frakN)$ with $\tau \in \Gal(k / \Q)$ interchanging the two real places and such that $\tau(\calO) = \calO$. Such a Galois-invariant order always exists. Now let $\calO^\prime = \mu \calO \mu^{-1}$, where $\tau(\mu) \mu^{-1} \not\in N_{B^\times}(\calO)$. Then $\tau(\calO^\prime) = \tau(\mu) \calO \tau(\mu)^{-1}$ is conjugate to $\calO^\prime$ under $\alpha = \tau(\mu) \mu^{-1}$ but $\tau(\calO^\prime) \neq \calO^\prime$. Therefore, the above construction applies. It would be interesting to understand a necessary and sufficient condition for this to be a torsion-free extension of $\Gamma_{st}$.
\end{rmk}
%%%%%%%%%%%%%%%%%%%%%%%%%

%%%%%%%%%%%%%%%%%%%%%%%%%
\subsection*{Maximal arithmetic lattices and torsion}
%%%%%%%%%%%%%%%%%%%%%%%%%

Let $\Gamma \leq G$ be a lattice. Then $\Gamma$ is contained in a maximal lattice \cite{BorelMaximality}, and our strategy is based upon enumerating maximal arithmetic lattices with bounded covolume, then finding which of these lattices contain a subgroup isomorphic to the fundamental group of a fake quadric. Therefore, we need to understand the maximal lattices in $G^*$, which are described as follows.

Suppose that the commensurability class of arithmetic lattices associated with the totally real field $k$ and $k$-quaternion algebra $B$ produces a fake quadric $X = \Gamma \backslash \calH$. Let $\Gamma_{st} = \Gamma \cap G^*$ be the stable subgroup of $\Gamma$. Then $\Gamma_{st}$ is contained in one of the maximal arithmetic lattices $\Gamma_{S, \calO}$ defined via Bruhat--Tits theory \cite{borel-commensurability} (see also Chinburg--Friedman \cite[\S 4]{Chinburg-Friedman-Selectivity} and Maclachlan--Reid \cite[\S 11.4]{Maclachlan--Reid}), where $S$ is a finite set of places of $k$, none of which ramify $B$. We briefly recall that $\Gamma_{S, \calO}$ is the image in $G^*$ of the group $N(\calE_S)$, where $\calE_S$ is the Eichler order of $B$ determined by $\calO$ and $S$ and $N(\calE_S)$ is the normalizer of $\calE_S$ in $B^\times$ (unless $S = \emptyset$, in which case $\Gamma_{S, \calO} = \Gamma_{\calO}$). Then $\Gamma_{S, \calO}$ is commensurable with $\Gamma_{\calO}$, and all maximal lattices in $G^*$ commensurable with $\Gamma_{\calO}$ come from this construction. We emphasize that $\Gamma_{S, \calO}$ is not uniquely determined by $S$.

Furthermore, the quotient $X = \Gamma \backslash \calH$ is smooth if and only if $\Gamma$ is torsion-free. Consequently, we need to understand torsion elements in $\Gamma$ in terms of arithmetic invariants of $B$. This comes from the following. For $s \in \Z_{\geq 1}$, we let $\zeta_s$ be a primitive root of unity of order $s$ (in an algebraic closure $\overline{k}$ of $k$).

Before stating our next result we require some additional terminology. Let $B$ be a $k$-quaternion algebra, $L$ a quadratic field extension of $k$ which embeds into $B$ and $\Omega\subset L$ a quadratic $\calO_k$-order. We say that $\Omega$ is \defi{selective} (with respect to $B$) if $\Omega$ does not embed into all maximal orders of $B$.
%%%%%%%%%%%%%%%%%%%%%%%%%
\begin{proposition}\label{prop:torsiontheorem}
Let $k$ be a totally real field and let $B$ be a quaternion division algebra over $k$ ramified at all but exactly two real places of $k$. Let $\calO \subset B$ be a maximal order and set $\Gamma_{\calO} = \iota(\calO^\times) \leq G$. Then the following are equivalent:
\begin{enumerate}

\item[\textup{(i)}] $\Gamma_{\calO}$ contains an element of order $m \geq 2$;

\item[\textup{(ii)}] $k(\zeta_{2m})$ embeds in $B$; 

\item[\textup{(iii)}] $k(\zeta_{2m}) / k$ is a quadratic extension in which no prime in $\Ram(B)$ splits.

\end{enumerate}
\end{proposition}
%%%%%%%%%%%%%%%%%%%%%%%%%

%%%%%%%%%%%%%%%%%%%%%%%%%
\begin{proof}
The equivalence of (ii) and (iii) is a standard local-global principle for $B$. To ensure that in fact we can take $\zeta_{2m} \in \calO^\times$, we need only rule out the possibility that $\calO_k[\zeta_{2m}]$ is selective with respect to $B$, and for this we appeal to a result of Chinburg--Friedman \cite[Theorem 3.3]{Chinburg-Friedman-Selectivity}. The extension $k(\zeta_{2m})/k$ is a totally imaginary quadratic extension of a totally real number field, and $B$ is unramified at a real place (in fact, precisely two). This rules out the possibility of selectivity and concludes our proof.
\end{proof}
%%%%%%%%%%%%%%%%%%%%%%%%%

We then have the following immediate consequence.

%%%%%%%%%%%%%%%%%%%%%%%%%
\begin{corollary}\label{corollary:23torsion}
Under the hypotheses of Proposition \ref{prop:torsiontheorem}, suppose further that $\Ram(B)$ contains no finite prime of $k$. Then $\Gamma_{\mc{O}}$ contains elements of orders $2$ and $3$.
\end{corollary}
%%%%%%%%%%%%%%%%%%%%%%%%%

Determining when an arbitrary lattice $\Gamma$ contains an element of finite order is considerably more nuanced.  We opt for the following necessary condition.  

\begin{lemma} \label{lem:embedtors}
The following statements hold.
\begin{enumerate}
\item[\textup{(a)}] The group $\calE_S^\times/\{\pm 1\}$ has an element of order $m \geq 2$ if and only if there exists an $\calO_k$-embedding $\calO_k[\zeta_{2m}] \hookrightarrow \calE_S$.
\item[\textup{(b)}] Let $T \subseteq L=k(\zeta_{2m})$ be an $\calO_k$-order of discriminant $\frakd_T \subseteq \calO_k$.  Then there exists \emph{no} $\calO_k$-embedding $T \hookrightarrow \calE_S$ if and only if one of the following two conditions hold:
\begin{enumerate}
\item[\textup{(i)}] There exists a prime $\frakp \in \Ram(B)$ that splits in $L$; or
\item[\textup{(ii)}] There exists a prime $\frakq \in S$ such that $\frakq \nmid \frakd_T$ and $\frakq$ is inert in $L$.
\end{enumerate}
\end{enumerate}
\end{lemma}

\begin{proof}
Part (a) is immediate.  For part (b), ruling out selectivity as in the proof of Proposition \ref{prop:torsiontheorem}, we need only check that there are local embeddings: for results in a general setup, see Voight \cite[Prop.\ 30.5.3, Lem.\ 30.6.16, Exer.\ 30.7]{voight:quatbook}.
\end{proof}

For stable maximal arithmetic lattices, there is a very precise (but technical) result due to Chinburg--Friedman \cite[Thm.\ 3.3, Thm.\ 3.6]{chinburg-friedman-finitesubgroups} indicating when $\Gamma_{S,\calO}$ contains an element of finite order.  Given the scope of our computation, we opt for the simpler check for torsion in $\calE_S^{\times}/\{\pm 1\} \leq \Gamma_{S,\calO}$ provided by Lemma \ref{lem:embedtors}.

\begin{comment}
The following states their result in the special case we need for this paper.

%%%%%%%%%%%%%%%%%%%%%%%%%
\begin{theorem}\label{theorem:CFmaximaltorsion}
Let $k$ be a totally real field, $B$ a quaternion algebra, and $\calO \subset B$ a maximal order. Then $\Gamma_{S, \calO}$ contains an element of order $m \geq 3$ if and only if the following conditions hold:

\begin{enumerate}

\item[\textup{(i)}] There is an $\calO_k$-embedding of $\calO_k[\zeta_m]$ into $\calO$. If $m = 2 \ell^r$ for some prime $ \ell \ge 3$, then $\calO_{k(\zeta_m)}$ also embeds into $\calO$.

\item[\textup{(ii)}] If $m=2\ell^r$ and $\frakl$ is a prime of $\calO_k$ over $\ell$ with $\frakl \not \in \Ram(B) \cup S$, then the absolute ramification index $e_\frakl$ is divisible by $\varphi(m)$, where $\varphi$ is the Euler totient function.

\item[\textup{(iii)}] For each prime $\frakp$ in $S$ at least one of the following holds:
\begin{itemize}

\item $\frakp$ splits in $k(\zeta_m) / k$;

\item $\frakp$ divides a rational prime $\ell \ge 3$ and $m = 2 \ell^r$ and $e_\frakp$ is not divisible by $\varphi(m)$; or

\item $\frakp$ divides a rational prime $\ell \ge 2$ and $m = \ell^r$.

\end{itemize}
\end{enumerate}
\end{theorem}
%%%%%%%%%%%%%%%%%%%%%%%%
\end{comment}

%%%%%%%%%%%%%%%%%%%%%%%%%
\section{Generalizing a volume inequality of Chinburg and Friedman}\label{ssec:chinburgfriedmanlemma43}
%%%%%%%%%%%%%%%%%%%%%%%%%

In this section, we revisit an inequality due to Chinburg--Friedman \cite[Lemma 4.3]{chinburg-smallestorbifold}, which provides a lower bound for the volume of an orbifold in the commensurability class determined by a quaternion algebra $B$ over a number field $k$ in terms of the arithmetic invariants of $B$. We expect that our improved bounds will be of use to those working on a wide variety of problems related to the volumes of arithmetic groups derived from quaternion algebras and therefore we prove a slightly more general result than is needed for the present paper.

For ease of comparing our results to those of Chinburg--Friedman, we recall their notation. Let $a, b \geq 0$ be non-negative integers with $a + b \geq 1$ and define $\calH_{a, b} = (\Htwo)^a \times (\mathbf{H}^3)^b$. The group $G_{a, b}=\PGL_2(\R)^a \times \PGL_2(\C)^b$ is the subgroup of the isometry group of $\calH_{a, b}$ preserving the factors and acting by orientation-preserving isometries on the $\mathbf{H}^3$ factors. Equip $\Htwo,\Hthree$ with the metric of constant curvature $-1$. Given an irreducible subgroup $\Gamma$ of $G_{a,b}$ with finite covolume, we let $\mu(X)$ denote the volume of the quotient orbifold $X = \Gamma \backslash \calH_{a,b}$.

We retain the notation from section 2, letting $k$ be a number field of degree $n=[k:\Q]$ with signature $(r_1,r_2)$ and $r_2=b$ and $a \leq r_1$. Let $B$ be a quaternion algebra defined over $k$ that is unramified at precisely $a$ real places of $k$. If $\calO$ is a maximal order of $B$ and $\Gamma_{\calO}$ the irreducible subgroup of $G_{a, b}$ defined by the image of $\calO^\times$ in $G_{a, b}$, we let $\scrC(k,B)$ denote the set of all subgroups of $G_{a,b}$ that are commensurable with $\Gamma_{\calO}$.

Recall that $\Ram(B)=\Ram_f(B) \cup \Ram_\infty(B)$, where $\Ram_f(B)$ is the subset of finite places of $k$ and $\Ram_\infty(B)$ the subset of archimedean places of $k$ ramifying in $B$. Let
\[
\omega_2(B) = \# \{ \frakp \in \Ram_f(B) : \N(\frakp) = 2 \}
\]
and $h(k,2,B)$ be the degree over $k$ of the maximal unramified elementary $2$-abelian Galois extension of $k$ in which all primes in $\Ram_f(B)$ split completely. Define
\begin{equation}\label{equation:mukbdef}
\mu(k,B) = \frac{2 d^{\frac{3}{2}} \zeta_k(2)}{2^{2 n + \omega_2(B)} \pi^{r_1 + n} [K(B) : k]},
\end{equation}
where $K(B)$ is the maximal $2$-elementary Galois extension of $k$ that is unramified outside $\Ram_\infty(B)$ in which all primes in $\Ram_f(B)$ split completely. We can replace $[K(B) : k]^{-1}$ with 
\begin{equation}\label{eq:ourfactor}
\frac{[\calO_k^\times : \calO_{k, +}^\times]}{2^{r_1} h(k, 2, B)},
\end{equation}
which is an integer multiple of $[K(B) : k]^{-1}$. We note that $[K(B):k]$ is the \defi{type number} of $B$; that is, the number of conjugacy classes of maximal orders of $B$ \cite[pp. 37, 39]{Chinburg-Friedman-Selectivity}.

The following result is due to Borel \cite{borel-commensurability} (cf.\ Chinburg--Friedman \cite[Prop.\ 2.1]{chinburg-smallestorbifold}).

%%%%%%%%%%%%%%%%%%%%%%%%%
\begin{proposition}\label{proposition:mukbbound}
If $\Gamma \in \scrC(k,B)$ and $X = \Gamma \backslash \calH_{a,b}$, then
\[
\mu(X) \geq \mu(k,B) (2\pi)^a \prod_{\substack{\frakp \in \Ram_f(B)\\ \N(\frakp) \neq 2}} \frac{\N(\frakp) -1}{2}.
\]
\end{proposition}
%%%%%%%%%%%%%%%%%%%%%%%%%

Chinburg--Friedman \cite[Lemma 4.3]{chinburg-smallestorbifold} show that
\[
\mu(X) > 0.69 \exp\left(0.37 n - \frac{19.08}{h(k,2,B)}\right),
\]
independent of $a,b$. The main result of this section is the following pair of lower bounds for $\mu(k,B)$. The first of the two bounds is most useful for ruling out fields of large degree, and is a strengthening of \cite[Lemma 4.3]{chinburg-smallestorbifold}. The second bound is useful for ruling out large $h(k,2,B)$ when $n$ is small, and has no analogue in previous work on volume bounds.

%%%%%%%%%%%%%%%%%%%%%%%%%
\begin{theorem}\label{theorem:lemma43generalization}
With notation as above, we have
\[
\mu(k,B) \geq \max_{(c_1, \dots, c_5)} c_1 \left( \frac{2}{w} \right) \Reg [\calO_k^\times : \calO_{k,+}^\times]\, \exp\left(c_2 n - c_3 r_1 + c_4 r_2 - \frac{c_5}{h(k,2,B)}\right),
\]
over
\[
(c_1, \dots, c_5) \in \left\{\begin{matrix}(1.785, 1.056, 1.139, 0.119, 19.075), \\ (2.116, 1.186, 1.080, 0.115, 6220.354) \end{matrix} \right\}.
\]
\end{theorem}
%%%%%%%%%%%%%%%%%%%%%%%%%

%%%%%%%%%%%%%%%%%%%%%%%%%
\begin{proof}
Since $h(k,2,B)$ divides the class number of $k$, it follows from the proof of the Brauer--Siegel theorem \cite[pp. 300, 322]{Lang-ANT} that for any real number $s > 1$ we have
\[
h(k,2,B) < \frac{w s (s - 1) \Gamma(s)^{r_2} \Gamma(\frac{s}{2})^{r_1} \zeta_k(s) d^{\frac{s}{2}}}{2^{r_1} \Reg 2^{r_2 s} \pi^{\frac{n s}{2}}}.
\]
Substituting this into \eqref{equation:mukbdef}, taking logarithms, and simplifying yields:
\begin{align}\label{equation:justtooklogs}
\log(\mu(k,B)) & \geq \log \left( \frac{2}{w} \cdot \Reg \cdot [\calO_k^\times : \calO_{k,+}^\times] \right) + \log \left( d^{\frac{3 - s}{2}} \cdot \frac{\zeta_k(2)}{\zeta_k(s)} \cdot 2^{-\omega_2(B)} \right) \notag \\ & + n \log \left( \frac{(2 \pi)^{\frac{s}{2}}}{4 \pi} \right) - r_1 \log (\pi 2^{\frac{s}{2}}) - \log (s (s - 1)) - \log \left( \Gamma \left( s \right)^{r_2} \Gamma \left( \frac{s}{2} \right)^{r_1} \right)
\end{align}

We now define a quantity $\widehat{T}(s,y)$ such that
\[
\log \left( d^{\frac{3 - s}{2}} \cdot \frac{\zeta_k(2)}{\zeta_k(s)} \cdot 2^{-\omega_2(B)} \right) + n \log \left( \frac{(2 \pi)^{\frac{s}{2}}}{4 \pi} \right) - r_1 \log (\pi 2^{\frac{s}{2}}) \ge n \widehat{T}(s, y)
\]
as in \cite[Prop.\ 3.1]{chinburg-smallestorbifold}. Our definition is for the most part analogous to the definition of $T(s,y)$ \cite[pp. 515]{chinburg-smallestorbifold}, though (in their notation) we are considering the case in which $s > 1$ and $\scrS = \emptyset$. In what follows $\gamma = 0.577156\dots$ is Euler's constant,
\[
\alpha(t) = \big( 3 t^{-3} (\sin(t) - t \cos(t)) \big)^2,
\]
$y > 0$ is a real variable,
\[
R(s,p_0) = \sum_{p>p_0\ \textrm{prime}} \log \left( \frac{1 - p^{-s}}{1 - p^{-2}} \right),
\]
and $K$ is the maximal elementary $2$-abelian unramified Galois extension of $k$ in which all prime ideals of $\Ram_f(B)$ split completely. Let $n_K$ denote the degree of $K$ over $\Q$ and write $n_K = r_1(K) + 2 r_2(K)$, where $r_1(K)$ is the number of real places of $K$ and $r_2(K)$ the number of complex places. We also assume $p$ is a prime not equal to $2$, $s > 1$, and define:
\begin{align*}
t(s, y, p, f, r) &= 2 (3 - s) \log(p^f) \sum_{m = 1}^\infty \frac{\alpha \left( \sqrt{y} \log(p^{m f}) \right)}{1 + p^{m f}} \\
&\quad - r \log(1 - p^{-2 f / r}) + r \log(1 - p^{-s f / r}) \\
\hat{q}_-(s, y, p, f, r) &= \min\{0, t(s, y, p, f, r)\} \\
\hat{j}(s,y,p) &= \min_{(f, r)} \frac{\hat{q}_-(s, y, p, f, r)}{f}
\end{align*}
In the definition of $\hat{j}$, the minimum is over integers $f$ such that $\N(\frakP_K) = p^f$ for a given prime $\frakP_K$ over $p$ and $r$ is the associated residue degree of $\frakP$ over the prime $\frakP \cap k$ of $k$. Notice that our assumption on $K$ having $2$-elementary abelian Galois group over $k$ implies that $r$ is either $1$ or $2$.

In the case where $p = 2$, there are two cases. We take:
\begin{align*}
\hat{q}_-(s, y, 2, 1, 1) &= \min\{0, t(s, y, 2, 1, 1) - \log(2)\} & \\
\hat{q}_-(s, y, 2, f, r) &= \min\{0, t(s, y, 2, f, r)\} & (f r \neq 1)\\
\hat{j}(s,y,2) &= \min_{(f, r)} \frac{\hat{q}_-(s, y, 2, f, r)}{f} & 
\end{align*}
with $t(s, y, 2, f, r)$ the same as above.

Finally, we fix a prime $p_0 > 2$ and define:
\begin{align*}
\widehat{T}(s, y) &= \frac{3 \gamma}{2} + \log (2 \sqrt{\pi}) + \frac{3}{2} \cdot \frac{r_1(K)}{n_K} - \frac{r_1}{n} \log(\pi) \\
& - s \left( \frac{\gamma + \log(2)}{2} + \frac{1}{2} \cdot \frac{r_1(K)}{n_K} + \frac{r_1}{n} \cdot \frac{\log(2)}{2} \right) \\
& - \frac{(3 - s)}{2} \int_0^\infty (1 - \alpha(x \sqrt{y})) \left( \frac{1}{\sinh(x)} + \frac{r_1(K)}{n_K} \cdot \frac{1}{2 \cosh^2(\frac{x}{2})} \right) dx \\
& - \frac{(3 - s) 6 \pi}{5 \sqrt{y} n_K} + R(s, p_0) + \sum_{p \leq p_0\ \textrm{prime}} \hat{j}(s,y,p).
\end{align*}
Then, equation (4.1) in \cite{chinburg-smallestorbifold}, which is equivalent to positivity of the sum appearing in the definition of $t$, gives
\[
\hat{j}(s, y, p) \ge \min_{(f, r)} \frac{r}{f} \log \left( \frac{1 - p^{-s f / r}}{1 - p^{- 2 f / r}} \right)
\]
(note the minor typo there, where $1.4$ should be replaced by $s$ in the right hand side). Notice that the second term is an increasing function of $f / r \in \Z$. As noted in the proof of \cite[Lem.\ 4.1]{chinburg-smallestorbifold}, if $f/r = 1$ never occurs and $3 \le p \le p_0$, then we can replace the above bound with
\[
\hat{j}(s, y, p) \ge \frac{1}{2} \log \left( \frac{1 - p^{-2 s}}{1 - p^{-4}} \right).
\]
for any such $p$. Lastly, $K$ has $[K : k] = h(k, 2, B)$ and $r_1(K) / n_K = r_1 / n$, so we have
\begin{align*}
\widehat{T}(s, y) &\ge \frac{3 \gamma}{2} + \log (2 \sqrt{\pi}) + \frac{r_1}{n}\left( \frac{3}{2} - \log(\pi) \right) \\
& - s \left( \frac{\gamma + \log(2)}{2} + \frac{r_1}{n} \cdot \frac{1 + \log(2)}{2} \right) \\
& - \frac{(3 - s)}{2} \int_0^\infty (1 - \alpha(x \sqrt{y})) \left( \frac{1}{\sinh(x)} + \frac{r_1}{n} \cdot \frac{1}{2 \cosh^2(\frac{x}{2})} \right) dx \\
& - \frac{(3 - s) 6 \pi}{5 \sqrt{y} n_K} + R(s, p_0) \\
& + \sum_{3 \le p \le p_0\ \textrm{prime}} \min\left\{ 0, q(s, y, p, 1, 1), \frac{1}{2} q(s, y, p, 2, 2), \frac{1}{2} \log \left( \frac{1 - p^{-2 s}}{1 - p^{-4}} \right) \right\} \\
& + \min \left\{0, \left\{ \frac{1}{f} q(s, y, 2, f, r) \right\}_{f / r \le 4}, \frac{1}{5} \log\left( \frac{1 - 2^{-5 s}}{1 - 2^{-10}} \right) \right\}.
\end{align*}
To get the two bounds in the statement of the theorem, take $s = 1.4$, $y = 0.1$, and $p_0 = 97$ for the first and $s = 1.35$, $y = 0.000001$, and $p_0 = 691$ for the second.
\end{proof}
%%%%%%%%%%%%%%%%%%%%%%%%%

Recall that in this paper we are interested in the case where $a=2$ and $b=0$, with the setting as in section 3. Proposition \ref{proposition:mukbbound} and Theorem \ref{theorem:lemma43generalization} then imply the following.

%%%%%%%%%%%%%%%%%%%%%%%%%
\begin{corollary}\label{corollary:lemma43productofplanes}
Let $k$ be a totally real field and let $B$ be a quaternion algebra over $k$ that is split at exactly two real places of $k$. If $\Gamma \in \scrC(k,B)$ and $X = \Gamma \backslash \calH$, then
\[
\mu(X) > \max_{(c_1, c_2, c_3)} c_1 \Reg [\calO_k^\times : \calO_{k,+}^\times]\, \exp\left(c_2 n - \frac{c_3}{h(k,2,B)}\right),
\]
where
\[
(c_1, c_2, c_3) \in \left\{ \begin{matrix} (70.497, -0.082, 19.075), \\ (83.552, 0.106, 6220.354) \end{matrix} \right\}
\]
\end{corollary}
%%%%%%%%%%%%%%%%%%%%%%%%%

As an application of Corollary \ref{corollary:lemma43productofplanes} we deduce our first upper bound for the degree $n$ of a totally real field such that there is a quaternion algebra $B$ over $k$ and $\Gamma \in \scrC(k,B)$ having covolume at most $32\pi^2$.

%%%%%%%%%%%%%%%%%%%%%%%%%
\begin{corollary}\label{corollary:nandhbounds}
With hypotheses as in Corollary \ref{corollary:lemma43productofplanes}, if further $\Gamma \in \scrC(k,B)$ has covolume at most $32\pi^2$, then $n \leq 38$. For each value of $n \in \{5, \dots, 38\}$ we have that $h(k,2,B)$ is bounded above by the quantity listed in \eqref{hbounds}.
\begin{equation} \label{hbounds}
\begin{array}{c|c}
n & h(k,2,B) \leq \\
\hline
\rule{0pt}{2.5ex} 
5 & 2^{14} \\
6 & 2^{12} \\
7 & 2^{15} \\
8 & 2^{12} \\
9 & 2^{10} \\
10 & 2^4 \\
11, 12 & 2^3 \\
13 \le n \le 16 & 2^2 \\
17 \le n \le 23 & 2^1 \\
24 \le n \le 38 & 2^0 
\end{array}
\end{equation}
\end{corollary}
%%%%%%%%%%%%%%%%%%%%%%%%%

%%%%%%%%%%%%%%%%%%%%%%%%%
\begin{proof}
Recall that $h(k, 2, B)$ divides the order of the $2$-part of the ideal class group of $k$, hence must be a power of two. Therefore, we can replace the numeric bound for $h(k, 2, B)$ coming from Corollary \ref{corollary:lemma43productofplanes} with the next smallest power of two. The assertion that $n \leq 38$ and the bounds on $h(k,2,B)$ for $7 \le n \le 38$ follow immediately from Corollary \ref{corollary:lemma43productofplanes}, the regulator bound $\Reg \geq 0.0062e^{0.738 n}$ due to Friedman \cite{Friedman}, and the trivial bound $[\calO_k^\times : \calO_{k,+}^\times] \geq 2$. For $n=5,6$, we use the stronger regulator bounds also due to Friedman \cite[Table 4]{Friedman}.
\end{proof}
%%%%%%%%%%%%%%%%%%%%%%%%%

We will need one further technique to bound the size of the normalizer. We will consider a quaternion algebra ramified at one more real place than $B$, and hence will obtain a Shimura curve. In this case, a bound on the area by the classical expression for the signature of a Fuchsian group implies a nontrivial bound on the size of the normalizer of the original lattice. We state our result only for algebras unramified at exactly two real places of $k$, but it is clear how one can generalize these methods to algebras unramified at more real places.

%%%%%%%%%%%%%%%%%%%%%%%%%
\begin{proposition}[The Fuchsian gambit]\label{prop:quatgambit}
Let $\Gamma \leq \Isom^h(\calH)$ be a lattice with associated quaternion algebra $B$ over $k$. Let $v=\covol(\Gamma)/(16\pi^2)$, so $v=1$ if $\Gamma$ is a fake quadric group.

Let $\Gamma_{st} = \Gamma \cap G^*$ be the stable subgroup of $\Gamma$.  Suppose that $\Gamma_{st}$ is contained in a maximal group $\Gamma^* \leq G^*$ associated to an Eichler order of (squarefree) discriminant $\frakd$, and let $\Gamma^1$ be the norm $1$ group associated to this order. Let $\frakp \mid \frakd$ and
\[
m=\frac{[\Gamma^* : \Gamma_{st}]}{[\Gamma : \Gamma_{st}]} \in \frac{1}{2}\Z.
\]
Let $\eps=2$ if $\frakp$ is a square in the narrow class group and $\eps=1$ otherwise.
\begin{enumerate}

\item[\textup{(a)}] Suppose that $\frakp$ is ramified in $B$. Then $[\Gamma^1 : \Gamma^*] \leq \epsilon 2^c \leq 2^{c+1}$, where
\[
c=\left\lfloor \frac{4v\eps}{m(\N\frakp-1)}\right\rfloor + 3 \leq \left\lfloor \frac{16v}{\N\frakp-1}\right \rfloor + 3.
\]
If $v=1$, then $c \leq 19$; if further $\N\frakp \geq 3$, then $c \leq 11$.

\item[\textup{(b)}] Suppose that $\frakp$ is unramified in $B$. Then $[\Gamma^1 : \Gamma^*] \leq 2^c$, where
\[
c=\left\lfloor \frac{4v(\N\frakp-1)}{m(\N\frakp+1)}\right\rfloor + 3
\]
and if $v=1$ then $c \leq 11$.
\end{enumerate}
\end{proposition}
%%%%%%%%%%%%%%%%%%%%%%%%%

%%%%%%%%%%%%%%%%%%%%%%%%%
\begin{proof}
% Let $\Gamma^*$ be a maximal stable group containing the fake quadric stable subgroup $\Gamma_{st}$, and let $\Gamma^1$ be the group of units of norm $1$ in an Eichler order corresponding to $\Gamma^*$. 
With the given notation, we have
\[
\covol(\Gamma^1) = \covol(\Gamma) \frac{[\Gamma : \Gamma_{st}]}{[\Gamma^* : \Gamma_{st}]} [\Gamma^* : \Gamma^1] = \frac{16\pi^2 v}{m} [\Gamma^* : \Gamma^1].
\]
% with $m=[\Gamma^* : \Gamma_{st}] / [\Gamma : \Gamma_{st}] \in \frac{1}{2}\Z$.

Now let $B'$ be the quaternion algebra over $k$ exchanging ramification at $\frakp$ for a real place $\infty_i$ and $\calO'$ be an Eichler order with the same level as $\calO$ (and maximal at all primes dividing the discriminant of $B$): so $\Ram(B') = \Ram(B) \cup \{\infty_i\} \setminus \{\frakp\}$.  If $\Delta^1$ is the group of units of norm $1$ in $\calO'$, then the volume formula implies that
\[
\coarea(\Delta^1) = \frac{\covol(\Gamma^1)}{(4 \pi)(\N \frakp - 1)} = \frac{4 \pi v}{m(\N \frakp - 1)} [\Gamma^* : \Gamma^1].
\]
Let $\Delta^*$ be a maximal lattice containing $\Delta^1$. From an explicit understanding of the normalizer group, something that only depends on the class group of $k$ and the level \cite[Proposition 1.16]{LV}, we have
\[
[\Delta^*:\Delta^1] = \frac{[\Gamma^*:\Gamma^1]}{\eps}
\]
where $\eps = 2$ if $\frakp$ is a square in the narrow class group of $k$ and $\eps=1$ otherwise. Therefore
\[
\coarea(\Delta^*) = \frac{\coarea(\Delta^1)}{[\Delta^* : \Delta^1]} = \frac{4 \pi v \eps}{m(\N \frakp - 1)}.
\]

Now for a Fuchsian group of area $\leq (2 \pi) a$, the formula for the coarea in terms of the signature
\[
a = 2 g - 2 + \sum_{i=1}^k \left(1 - \frac{1}{e_i}\right)
\]
determines the maximum $2$-rank of $\Delta^* / \Delta^{*2}$ with signature $(0;2,\dots,2)$, which is $k - 1$ where
\[
k = \lfloor 2(a + 2) \rfloor.
\]
However, $\Delta^* / \Delta^1$ is an elementary abelian $2$-group, so
\[
\log_2 [\Delta^* : \Delta^1] \leq \left\lfloor 2(a + 2) \right\rfloor - 1 = \left\lfloor 2\left(\frac{2v \eps}{m(\N \frakp - 1)} + 2 \right) \right\rfloor - 1 = \left\lfloor \frac{4v \eps}{m(\N \frakp - 1)} \right\rfloor + 3.
\]
The other calculations for part (a) are similar.

For part (b), we instead take $B'$ to be the quaternion algebra augmenting the ramification set of $B$ by both $\frakp$ and a real place $\infty_i$, so now $\Ram(B')=\Ram(B) \cup \{\frakp,\infty_i\}$ and we take $\calO'$ to be an Eichler order with level $\frakN / \frakp$. Now we have
\[
\coarea(\Delta^1) = \frac{\covol(\Gamma^1)}{(4 \pi)}\frac{\N \frakp - 1}{\N \frakp + 1} = \frac{4 \pi v(\N \frakp - 1)}{m(\N \frakp + 1)} [\Gamma^* : \Gamma^1].
\]
The result follows from a similar argument to the first part.
\end{proof}

%%%%%%%%%%%%%%%%%%%%%%%%%
\section{Enumerating fields and fake quadrics}\label{sec:enumeration}
%%%%%%%%%%%%%%%%%%%%%%%%%

We now describe Borel's formula for the covolume of $\Gamma_{S, \calO} \backslash \calH$. Let $\omega_2(B)$, $K(B)$, and $h(k,2,B)$ be as above. Recall that $\Htwo$ is equipped with its metric of constant curvature $-1$ and $\calH$ with the product metric. For a lattice $\Gamma \in \Isom(\calH)$, by $\vol(\Gamma \backslash \calH)$ or $\covol(\Gamma)$ we always mean the measure of $\Gamma \backslash \calH$ in this metric.

Let $\Gamma_{S, \calO} \leq G^* = \PGL_2(\mathbb R) \times \PGL_2(\mathbb R)$. Borel \cite{borel-commensurability} (see also Chinburg--Friedman \cite[Prop. 2.1]{chinburg-smallestorbifold}) showed that
\begin{equation}\label{equation:volumeformulaequality}
\covol(\Gamma_{S, \calO}) = \frac{8 \pi^2 d^{\frac{3}{2}} \zeta_k(2)}{2^\fraks (4\pi^2)^{n} [K(B) : k]} \prod_{\frakp \in \Ram_f(B)} \frac{\N(\frakp) - 1}{2} \prod_{\frakp \in S}(\N(\frakp) + 1)
\end{equation}
for some integer $\fraks$ with $0\le \fraks \le \# S$; the integer $\fraks$ is given explicitly by class field theory, as explained by Maclachlan--Reid \cite[p.\ 356]{Maclachlan--Reid}. It follows that when $S = \emptyset$, the group $\Gamma_{S, \calO} = \Gamma_{\calO}$ has minimal covolume in its commensurability class. Following Chinburg--Friedman \cite[pp. 512]{chinburg-smallestorbifold}, we see that
\begin{equation}\label{equation:volumeinequality}
\covol(\Gamma_{\calO}) \ge \frac{8 \pi^2 d^{3/2} \zeta_k(2)}{(4 \pi^2)^{n} 2^{m} h(k,2,B)} \prod_{\frakp \in \Ram_f(B)} \frac{\N(\frakp) - 1}{2}.
\end{equation}

In this section, we prove some further restrictions on the totally real field $k$ associated with an irreducible fake quadric, then we explain our methods for computing the possible commensurability classes of irreducible arithmetic subgroups of $\Isom^+(\calH)$ containing a fake quadric.

%%%%%%%%%%%%%%%%%%%%%%%%%
\subsection*{Further restrictions on $k$}
%%%%%%%%%%%%%%%%%%%%%%%%%

Let $k$ be a totally real field and $B$ be a quaternion algebra defined over $k$ that is unramified at exactly two real places of $k$. In this section we will enumerate the totally real fields $k$ for which $\scrC(k,B)$ contains a group with covolume less than $32\pi^2$. We prove the following.

%%%%%%%%%%%%%%%%%%%%%%%%%
\begin{theorem}\label{theorem:degreebounds}
If $\Gamma\in\mathscr C(k,B)$ is torsion-free, stable, and has covolume a submultiple of $32\pi^2$, then $n\leq 8$ and the root discriminant $\delta$ of $k$ satisfies the bound listed in \eqref{rootdiscbounds}.
\begin{equation} \label{rootdiscbounds}
\begin{array}{c|c}
n & \delta\leq \\
\hline
\rule{0pt}{2.5ex} 
2 & 118.436 \\
3 & 118.436 \\
4 & 92.754 \\
5 & 31.823 \\
6 & 15.986 \\
7 & 15.269 \\
8 & 14.262 \\
\end{array}
\end{equation}
\end{theorem}
%%%%%%%%%%%%%%%%%%%%%%%%%

We begin the proof of Theorem \ref{theorem:degreebounds} by showing that $n \leq 8$. Note that it already follows from Corollary \ref{corollary:nandhbounds} that $n\leq 38$. We handle the most difficult cases, where $n \le 10$, and leave the proof for larger degrees, which follow the same lines of reasoning, to the reader. (See work of Linowitz--Voight \cite{LV} for an earlier application and further elaboration for the combined techniques employed here.)

%%%%%%%%%%%%%%%%%%%%%%%%%
\begin{proposition}\label{proposition:not10}
If $\Gamma \in \scrC(k,B)$ has covolume less than $32\pi^2$, then $n \neq 10$.
\end{proposition}
%%%%%%%%%%%%%%%%%%%%%%%%%

%%%%%%%%%%%%%%%%%%%%%%%%%
\begin{proof}
If $n = 10$ then Corollary \ref{corollary:nandhbounds} shows that $h(k,2,B) \leq 2^4=16$. Combining this with the trivial bounds $\omega_2(B)\leq 10$, $m \leq 9$, $\zeta_k(2) \geq (4/3)^{\omega_2(B)}$ and 	\[
	\prod_{\frakp \in \Ram_f(B)} \frac{\N(\frakp) - 1}{2} \geq \left( \frac{1}{2} \right)^{\omega_2(B)},
	\] we see from \eqref{equation:volumeinequality} that $\delta \leq 30.386$. The discriminant bounds of Odlyzko \cite{Odlyzko-bounds} and Poitou \cite{Poitou} (see also Brueggeman--Doud \cite[Theorem 2.4]{doud}) imply that a totally real field of degree $10$ with at least $6$ primes of norm $2$ must have root discriminant at least $33.498$, hence $\omega_2(B)\leq 5$. Applying \eqref{equation:volumeinequality} again shows that $\delta \leq 26.545$. A totally real field of degree at least $40$ must have root discriminant at least $30.890$, hence the class number $h$ of $k$ satisfies $h \leq 3$. As $h(k,2,B) \leq h$ and $h(k,2,B)$ is a power of two, we have $h(k,2,B)\leq 2$. Applying the arguments above once more shows that $\omega_2(B) \leq 3$ and $\delta \leq 21.892$. 

In order to proceed we will apply the following theorem of Armitage--Fr{\"o}hlich \cite{armitage-frolich}.

%%%%%%%%%%%%%%%%%%%%%%%%%
\begin{theorem}[Armitage--Fr{\"o}hlich]\label{theorem:af}
Let $h_2$ denote the rank of the $2$-part of the ideal class group of $k$ and $m$ be defined as above. Then
\[
m \le \left\lfloor \frac{n}{2} \right\rfloor + h_2.
\]
\end{theorem}
%%%%%%%%%%%%%%%%%%%%%%%%%

An immediate consequence of Theorem \ref{theorem:af} is that $m \leq 6$, hence $\delta \leq 19.058$. A totally real field of degree at least $20$ has root discriminant at least $21.401$, hence $h = h(k,2,B) = 1$. Therefore $m \leq 5$ and $\delta \leq 17.376$. Recall that if $k$ is a totally real field of degree $n$, class number $h$ and $2$-rank of totally positive units modulo squares equal to $m$, then the narrow class field of $k$ has degree $n h 2^{m}$. If $m=5$ then applying the Odlyzko bounds to the narrow class field, we would have $\delta \geq 19.382$, whereas \eqref{equation:volumeinequality} implies $\delta\leq 17.376$. This contradiction shows that $m\leq 4$. The same argument shows that $m\leq 3$, hence $\delta \leq 15.842$. A totally real field of degree $10$ with at least $2$ primes of norm $2$ must have root discriminant at least $17.607$, hence $\omega_2(B) \leq 1$ and $\delta \leq 15.008$. Repeating this argument once more shows that $\omega_2(B) = 0$ and $\delta \leq 13.949$. Voight \cite{voight-fields} showed that there are no totally real number fields of degree $10$ and root discriminant less than $14$, hence $n\neq 10$.
\end{proof}
%%%%%%%%%%%%%%%%%%%%%%%%%

%%%%%%%%%%%%%%%%%%%%%%%%%
\begin{proposition}\label{proposition:not9}
If $\Gamma \in \scrC(k,B)$ is torsion-free and has covolume a submultiple of $32\pi^2$ then $n\neq 9$.
\end{proposition}
%%%%%%%%%%%%%%%%%%%%%%%%%

%%%%%%%%%%%%%%%%%%%%%%%%%
\begin{proof}
We know from Corollary \ref{corollary:nandhbounds} that $h(k,2,B) \leq 2^{10}=1024$. The trivial bounds $\omega_2(B) \leq 9$, $m \leq 8$ and $\zeta_k(2) \geq (4/3)^{\omega_2(B)}$ show, via \eqref{equation:volumeinequality}, that $\delta \leq 42.424$. A totally real field of degree at least $288=9\cdot 2^5$ must have root discriminant at least $50.176$, hence $h(k,2,B) \leq h\leq 2^4$ and $\delta \leq 31.176$. A totally real field of degree $9$ containing at least $6$ primes of norm $2$ must have root discriminant at least $33.182$, hence $\omega_2(B)\leq 5$ and $\delta\leq 27.647$. Applying the Odlyzko bounds to the Hilbert class field of $k$ shows that $h(k,2,B) \leq h \leq 2$. Theorem \ref{theorem:af} now implies that $m \leq 5$, allowing us to conclude that $\delta \leq 20.317$. Repeating these arguments and applying the Odlzyko bounds to the narrow class field of $k$ (a number field of degree $9 \cdot 2^{m}$) shows that $\omega_2(B) \leq 1$, $h = 1$, $m \leq 3$ and $\delta \leq 15.445$.

Let $\Gamma^* \in \mathscr C(k,B)$ be a maximal arithmetic subgroup containing $\Gamma$ and suppose that $\Gamma^*$ has covolume at most $32\pi^2/3$. In this case \eqref{equation:volumeinequality} shows that $\delta \leq 14.238$. All totally real fields of degree $9$ and root discriminant less than $15$ were enumerated by Voight \cite{voight-tables}, and an easy computation shows that all satisfy $m\leq 1$. This improves our bound to $\delta \leq 12.849$. As every totally real field of degree $9$ has root discriminant at least $12.869$, we conclude that the covolume of $\Gamma^*$ is either $32\pi^2$ or $16\pi^2$. To conclude our proof it now suffices to show that $\Gamma^*$ contains an element of order $3$.

We first show that $B$ must be ramified at precisely one prime of $k$ of norm $2$ and at no other finite primes. Suppose not. Poitou's bounds imply that a totally real field of degree $9$ containing primes of norms $2$ and $3$ must have root discriminant at least $15.685$. Since we already showed that $\delta \leq 15.445$, we conclude that either $B$ ramifies at a prime of norm $2$ and two primes of norm at least $4$, or else $B$ ramifies at no primes of norm $2$ and a prime of norm at least $3$.

If $B$ ramifies at a prime of norm $2$, we conclude from \eqref{equation:volumeinequality} that $\delta \leq 14.545$. We already saw that all totally real fields of degree $9$ with root discriminant less than $15$ satisfy $m \leq 1$. This allows us to improve our bound to $\delta \leq 13.938$. No totally real field of degree $9$ with root discriminant less than $13.938$ contains a prime of norm $2$, which is a contradiction.

Now, suppose that $B$ ramifies at no primes of norm $2$. It therefore ramifies at a prime of norm at least $3$. Equation \eqref{equation:volumeinequality} shows that $\delta \leq 14.988$, hence $m \leq 1$ (as noted above). This improves our bound to $\delta \leq 13.525$. There are only two totally real fields of degree $9$ with root discriminant this small, and neither contains a prime of norm less than $19$. Therefore $\delta \leq 11.494$. However, there are no totally real fields of degree $9$ and root discriminant less than $11.494$. This contradiction proves our claim that $B$ is ramified at precisely one prime of norm $2$ and no other finite primes of $k$.

We now show that $\Gamma^*$ contains an element of order $3$. Since $\Gamma^*$ is a maximal arithmetic subgroup it is of the form $\Gamma_{S,\calO}$ for some finite set $S$ of primes of $k$ (disjoint from $\Ram_f(B)$) and some maximal order $\calO$ of $B$. We claim that $S = \emptyset$. Indeed, if $S$ is not empty then the Poitou bounds yield a contradiction to the fact that $\delta \leq 15.445$ unless $S$ contains a prime of norm at least $4$. Equations \eqref{equation:volumeformulaequality} and \eqref{equation:volumeinequality} allow us to find a contradiction by arguing exactly as above. We can therefore assume that $\Gamma^* = \Gamma_{\calO}$ for some maximal order $\calO$ of $B$. Since $B$ is ramified at precisely one prime of norm $2$ and no other finite primes of $k$, our claim that $\Gamma^*$ contains an element of order $3$ follows from Theorem \ref{prop:torsiontheorem} and the fact that $2$ is inert in $\Q(\zeta_3) / \Q$.
\end{proof}
%%%%%%%%%%%%%%%%%%%%%%%%%

%%%%%%%%%%%%%%%%%%%%%%%%%
\begin{proposition}\label{proposition:not8}
If $n=8$ and $\Gamma \in \scrC(k,B)$ has covolume at most $32\pi^2$ then $\delta\leq 14.262$.
\end{proposition}
%%%%%%%%%%%%%%%%%%%%%%%%%

%%%%%%%%%%%%%%%%%%%%%%%%%
\begin{proof}
We know from Corollary \ref{corollary:nandhbounds} that $h(k,2,B) \leq 2^{12}=4096$, hence $[K(B):k]\leq 2^{m}h(k,2,B)\leq 2^7\cdot 2^{12}=2^{19}$. From this we conclude, via \eqref{equation:volumeformulaequality} and the trivial estimates $\omega_2(B) \leq 8$ and $\zeta_k(2) \geq (4/3)^{\omega_2(B)}$, that $\delta\leq 51.102$. The Odlyzko bounds imply that every totally real number field of degree at least $8\cdot 2^6$ has root discriminant greater than $53.494$, hence $h(k,2,B)\leq 2^5$ and $[K(B):k]\leq 2^{12}$. Assume that $[K(B):k]=2^{12}$ and recall that $K(B)$ is unramified at all places of $k$ not lying in $\Ram_\infty(B)$. In particular there are at least two real places of $B$ which are unramified in the extension $K(B)/k$, hence $K(B)$ has at least $2\cdot [K(B):k]$ real places. Applying the Odlyzko bounds to this field (which has the same root discriminant as $k$) shows that $\delta>62.785$ unless $\omega_2(B)\leq 2$. When $\omega_2(B)\leq 2$ equation \eqref{equation:volumeformulaequality} shows that $\delta\leq 27.848$. Every number field of degree $8\cdot 2^{12}$ with at least $2\cdot 2^{12}$ real places has root discriminant greater than $28.535$. This shows that $[K(B):k]\neq 2^{12}$. An identical argument shows that $[K(B):k]\leq 2$, $\omega_2(B)\leq 1$ and $\delta\leq 14.262$.\end{proof}
%%%%%%%%%%%%%%%%%%%%%%%%%

%%%%%%%%%%%%%%%%%%%%%%%%%
\begin{proposition}\label{proposition:not7}
If $n=7$ and $\Gamma \in \scrC(k,B)$ has covolume at most $32\pi^2$ then $\delta\leq 15.269$.
\end{proposition}
%%%%%%%%%%%%%%%%%%%%%%%%%

%%%%%%%%%%%%%%%%%%%%%%%%%
\begin{proof}
We know from Corollary \ref{corollary:nandhbounds} that $h(k,2,B) \leq 2^{15}=32768$, hence $[K(B):k]\leq 2^{m}h(k,2,B)\leq 2^6\cdot 2^{15}=2^{21}$. From this we conclude, via \eqref{equation:volumeformulaequality} and the trivial estimates $\omega_2(B) \leq 7$ and $\zeta_k(2) \geq (4/3)^{\omega_2(B)}$, that $\delta\leq 69.348$. If $h(k,2,B)=2^{15}$ then the class field associated to $h(k,2,B)$ is totally real of degree $7\cdot 2^{15}$ and has at least $\omega_2(B)\cdot 2^{15}$ primes of norm $2$. If $\omega_2(B)\geq 1$ then this field has root discriminant $\delta\geq 78.350$, a contradiction. On the other hand, if $\omega_2(B)=0$ then \eqref{equation:volumeformulaequality} shows that $\delta\leq 52.923$ while applying the Odlyzko bounds to the class field associated to $h(k,2,B)$ shows that $\delta>60.702$. Therefore $h(k,2,B)\leq 2^{14}$ and $[K(B):k]\leq 2^{20}$. Arguing as above and applying the Odlyzko-Poitou bounds to $K(B)$ show that $[K(B):k]\leq 4$ and $\delta\leq 15.269$.\end{proof}
%%%%%%%%%%%%%%%%%%%%%%%%%

%%%%%%%%%%%%%%%%%%%%%%%%%
\begin{proposition}\label{proposition:not6}
If $n=6$ and $\Gamma \in \scrC(k,B)$ has covolume at most $32\pi^2$ then $\delta\leq 15.986$.
\end{proposition}
%%%%%%%%%%%%%%%%%%%%%%%%%

%%%%%%%%%%%%%%%%%%%%%%%%%
\begin{proof}
We know from Corollary \ref{corollary:nandhbounds} that $h(k,2,B) \leq 2^{12}=4096$, hence $[K(B):k]\leq 2^{m}h(k,2,B)\leq 2^5\cdot 2^{12}=2^{17}$. From this we conclude, via \eqref{equation:volumeformulaequality} and the trivial estimates $\omega_2(B) \leq 6$ and $\zeta_k(2) \geq (4/3)^{\omega_2(B)}$, that $\delta\leq 65.636$. If $h(k,2,B)=2^{12}$ then the class field associated to $h(k,2,B)$ is totally real of degree $6\cdot 2^{12}$ and has at least $\omega_2(B)\cdot 2^{12}$ primes of norm $2$. If $\omega_2(B)\geq 1$ then this field has root discriminant $\delta\geq 81.118$, a contradiction. On the other hand, if $\omega_2(B)=0$ then \eqref{equation:volumeformulaequality} shows that $\delta\leq 50.090$ while applying the Odlyzko bounds to the class field associated to $h(k,2,B)$ shows that $\delta>60.241$. Therefore $h(k,2,B)\leq 2^{11}$ and $[K(B):k]\leq 2^{16}$. Identical arguments show that $[K(B):k]\leq 4$ and $\delta\leq 15.986$.\end{proof}
%%%%%%%%%%%%%%%%%%%%%%%%%

%%%%%%%%%%%%%%%%%%%%%%%%%
\begin{proposition}\label{proposition:not5}
If $n=5$ and $\Gamma \in \scrC(k,B)$ has covolume at most $32\pi^2$ then $\delta\leq 31.823$.
\end{proposition}
%%%%%%%%%%%%%%%%%%%%%%%%%

%%%%%%%%%%%%%%%%%%%%%%%%%
\begin{proof}
We know from Corollary \ref{corollary:nandhbounds} that $h(k,2,B) \leq 2^{14}=16384$, hence $[K(B):k]\leq 2^{m}h(k,2,B)\leq 2^4\cdot 2^{14}=2^{18}$. From this we conclude, via \eqref{equation:volumeformulaequality} and the trivial estimates $\omega_2(B) \leq 5$ and $\zeta_k(2) \geq (4/3)^{\omega_2(B)}$, that $\delta\leq 96.468$. If $h(k,2,B)=2^{14}$ then the class field associated to $h(k,2,B)$ is totally real of degree $5\cdot 2^{14}$ and has at least $\omega_2(B)\cdot 2^{14}$ primes of norm $2$. If $\omega_2(B)=5$ then applying the Odlyzko bounds to the class field associated to $h(k,2,B)$ shows that $\delta\geq 503.890$, a contradiction. Similar arguments show that we must have $\omega_2(B)=0$, in which case $\delta\leq 73.619$. Since $B$ must ramify at a finite number of primes having even cardinality, a finite prime of $k$ must ramify in $B$. The Odlyzko bounds, applied once again to the class field associated to $h(k,2,B)$, show that this prime must have norm at least $11$. Equation \eqref{equation:volumeformulaequality} now shows that $\delta\leq 59.401$. This is a contradiction, as the Odlyzko bounds show that a totally real number field of degree $5\cdot 2^{14}$ must have root discriminant at least $60.571$. This shows that $h(k,2,B)\leq 2^{13}$. Identical arguments (for $h(k,2,B)\in\{2^x : x\leq 13\}$) show that $h(k,2,B)\leq 2^2$, in which case $\delta\leq 31.823$. \end{proof}
%%%%%%%%%%%%%%%%%%%%%%%%%

%%%%%%%%%%%%%%%%%%%%%%%%%
\begin{proposition}\label{proposition:not4}
If $n=4$ and $\Gamma \in \scrC(k,B)$ is torsion-free, stable, and has covolume at most $32\pi^2$ then $\delta\leq 92.754$.
\end{proposition}
%%%%%%%%%%%%%%%%%%%%%%%%%

%%%%%%%%%%%%%%%%%%%%%%%%%
\begin{proof}
We begin with the case in which $B$ is ramified at a finite prime $\mathfrak p$ of $k$. Let $\Gamma^*$ be a maximal arithmetic group containing $\Gamma$ and $\Gamma^1$ be the group of norm one units in the Eichler order associated to $\Gamma^*$. Proposition \ref{prop:quatgambit} shows that $\covol(\Gamma^1)\leq 2^{12}\cdot 32\pi^2$ if $N\mathfrak p\geq 3$ and that $\covol(\Gamma^1)\leq 2^{20}\cdot 32\pi^2$ when $N\mathfrak p=2$. In the former case \cite[\S 7.3]{borel-commensurability} shows that \[2^{12} \geq \frac{d^{3/2}}{(4\pi^2)^{n}},\] whereas in the latter case we obtain (employing the estimate $\zeta_k(2)\geq 4/3$, which follows from the Euler product expansion of $\zeta_k(s)$) \[2^{20} \geq \frac{d^{3/2}(4/3)}{(4\pi^2)^{n}}.\] In the former case we deduce that $\delta\leq 46.377$ and in the latter case we have $\delta\leq 111.391$. Since a quaternion algebra over $k$ must ramify at an even number of primes, there exists a prime $\mathfrak q\neq \mathfrak p$ such that $B$ ramifies at $\mathfrak q$ as well. If $N\mathfrak q>2$ then we have $\delta\leq 46.377$ as above. Suppose now that $N\mathfrak q=2$. In this case the class field associated to $h(k,2,B)$ is totally real of degree $4\cdot h(k,2,B)$ and has at least $2h(k,2,B)$ primes of norm $2$. The Odlyzko-Poitou bounds imply that $h(k,2,B)\leq 2^4$, in which case \eqref{equation:volumeformulaequality} implies that $\delta\leq 42.972$.

Finally, consider the case in which $B$ is unramified at all finite primes of $k$. In this case Corollary \ref{corollary:23torsion} shows that $\Gamma_{\mathcal O}$ contains elements of orders $2$ and $3$. If there exists a maximal order $\mathcal O$ of $B$ such that $\Gamma\subset \Gamma_{\mathcal O}$ then because $\Gamma$ is torsion-free, the covolume of $\Gamma_{\mathcal O}$ is at most $32\pi^2/6$. If $\Gamma\not\subset \Gamma_{\mathcal O}$ then $\Gamma$ must be contained in $\Gamma_{S,\mathcal O}$ for some maximal order $\mathcal O$ of $B$ and non-empty set $S$ of primes of $k$. We claim that $\covol(\Gamma_{\mathcal O})\leq 16\pi^2$ except possibly when $S=\{\mathfrak p\}$ and $N\mathfrak p=2$. Indeed, this follows from \eqref{equation:volumeformulaequality} and the fact that $\covol(\Gamma)\leq 32\pi^2$. In all of the cases considered thus far, $\covol(\Gamma_{\mathcal O})\leq 16\pi^2$, hence Corollary \ref{corollary:lemma43productofplanes} shows that $h(k,2,B)\leq 2^{14}$ and $\delta\leq 92.754$ by \eqref{equation:volumeinequality}.

We now consider the remaining case in which $\Gamma\subset \Gamma_{S,\mathcal O}$ with $S=\{\mathfrak p\}$ and $N\mathfrak p=2$. Proposition \ref{prop:quatgambit}(b) shows that in this case $\covol(\Gamma^1)\leq 2^{11}\cdot 32\pi^2$ and the proposition follows from the volume formula for $\Gamma^1$ \cite[\S 7.3]{borel-commensurability} as above. \end{proof}
%%%%%%%%%%%%%%%%%%%%%%%%%

%%%%%%%%%%%%%%%%%%%%%%%%%
\begin{proposition}\label{proposition:nandhboundssmalldegrees}
If $n\in\{2,3\}$ and $\Gamma\in\mathscr C(k,B)$ has covolume at most $32\pi^2$ then $\delta\leq 118.436$.
\end{proposition}
%%%%%%%%%%%%%%%%%%%%%%%%%

%%%%%%%%%%%%%%%%%%%%%%%%%
\begin{proof}
If $k$ is a real quadratic field or a totally real cubic field, then $h\leq \frac{1}{2}\sqrt{d}$: this follows from work of \cite{Le}, Ramare \cite[Corollary 2]{Ramare} for the quadratic case, and Louboutin \cite[Corollary 4]{Louboutin-classnumberbounds} for the cubic case. Recalling that $\omega_2(B)$ is the number of primes of $k$ that ramify in $B$ and have norm $2$, it is clear that
\[
\prod_{\frakp \in \Ram_f(B)} \frac{\N(\frakp) - 1}{2} \geq \left( \frac{1}{2} \right)^{\omega_2(B)}.
\]
By considering the Euler product expansion of $\zeta_k(s)$, one sees that $\zeta_k(2) \geq (4/3)^{\omega_2(B)}$. We also note the trivial estimates $\omega_2(B)\leq n$, $m\leq n-1$, $h(k,2,B)\leq h$. Combining all of these estimates with \eqref{equation:volumeinequality} proves the proposition.
\end{proof}
%%%%%%%%%%%%%%%%%%%%%%%%%

%%%%%%%%%%%%%%%%%%%%%%%%%
\subsection*{Enumerating maximal arithmetic subgroups containing fake quadrics}
%%%%%%%%%%%%%%%%%%%%%%%%%

We now know that the field of definition of a fake quadric is a totally real field of degree $2 \leq n\leq 8$. Now we describe how one enumerates all maximal arithmetic subgroups of $\Isom(\calH)$ that can contain the stable subgroup of the fundamental group of a fake quadric. All of the computations described in this section were carried out with \textsc{Magma} \cite{Magma}.

Recall that a maximal arithmetic subgroup $\Gamma_{S, \calO}$ of $\PGL_2(\mathbb R) \times \PGL_2(\mathbb R)$ defined over $k$ has covolume 
\begin{equation}\label{equation:volumeformulaequalityagain}
\covol(\Gamma_{S, \calO}) = \frac{8 \pi^2 d^{\frac{3}{2}} \zeta_k(2) \left( \prod_{\frakp \in \Ram_f(B)} \frac{\N(\frakp) - 1}{2} \right) \prod_{\frakp \in S}(\N(\frakp) + 1)}{2^\fraks (4 \pi^2)^{n} [K(B):k]},
\end{equation}
for some integer $\fraks$ with $0\le \fraks \le |S|$. All of these quantities are easily computed with \textsc{Magma}. We remark that because $K(B)$ is the maximal abelian extension of $k$ with $2$-elementary Galois group that is unramified outside of $\Ram_\infty(B)$ and in which all primes of $\Ram_f(B)$ split completely, we can use computational class field theory to compute $[K(B):k]$.

It is trivial to enumerate all real quadratic fields with root discriminant less than $118.436$, and in degrees $3,\dots,8$ all of the totally real fields satisfying the root discriminant bounds \eqref{rootdiscbounds} were enumerated by Voight \cite{voight-fields}. Fix one such field $k$. Using \eqref{equation:volumeformulaequalityagain} we compute all possible sets $\Ram(B)$ and $S$ for which the right hand side of \eqref{equation:volumeformulaequalityagain} is at most $32\pi^2$. Here $\Ram(B)$ is a set of even cardinality consisting of finite primes of $k$ and $n - 2$ real places of $k$, and $S$ is a set of finite primes of $k$ none of which lie in $\Ram(B)$.

If $\Gamma_{st}$ is the stable subgroup of the fundamental group $\Gamma$ of a fake quadric, then $\Gamma_{st}$ is contained in a maximal arithmetic group $\Gamma_{S,\mathcal O}$ (where $\mathcal O$ is a maximal order of the quaternion algebra $B$ over $k$ defined by the ramification set $\Ram(B)$). Then $\covol(\Gamma_{st}) = 32 \pi^2$ if $\Gamma_{st}$ is properly contained in $\Gamma$ and $\covol(\Gamma_{st}) = 16 \pi^2$ if $\Gamma=\Gamma_{st}$. In particular:
\begin{equation}
\covol(\Gamma_{S,\mathcal O}) = \begin{cases} \displaystyle \frac{32\pi^2}{[\Gamma_{S,\mathcal O}:\Gamma]} & \Gamma \neq \Gamma_{st} \\ & \\ \displaystyle \frac{16\pi^2}{[\Gamma_{S,\mathcal O}:\Gamma]} & \Gamma = \Gamma_{st} \end{cases}
\end{equation}
It follows that we need only consider those sets $\Ram(B)$ and $S$ for which the associated group $\Gamma_{S,\mathcal O}$ has covolume an integral submultiple of $32\pi^2$.

We note that in many cases, Theorem \ref{thm:galoisdescswap} implies that $\Gamma$ cannot properly contain $\Gamma_{st}$. For instance, this is necessarily the case when $\#\Aut(k / \mathbb Q) = 1$. In these cases it follows that $\Gamma = \Gamma_{st}$, $\covol(\Gamma_{st}) = 16 \pi^2$ and $\covol(\Gamma_{S,\mathcal O}) = 16 \pi^2 / [\Gamma_{S,\mathcal O}:\Gamma]$.

If $\calO$ and $\calO^\prime$ are maximal orders of $B$, then the groups $\Gamma_{S, \calO}$ and $\Gamma_{S, \calO^\prime}$ will not always be conjugate, but they have equal covolumes. If $\calO$ and $\calO^\prime$ are conjugate by an element of $B^\times$, then $\Gamma_{S, \calO}$ and $\Gamma_{S, \calO^\prime}$ will be conjugate \cite[Lemma 4.1]{Chinburg-Friedman-Selectivity}. Also, the number of conjugacy classes of maximal orders of $B$ is equal to $[K(B):k]$ \cite[pp. 37]{Chinburg-Friedman-Selectivity}, and this quantity is referred to as the \defi{type number} $t_B$ of $B$. Using the methods described above, we see that all but four of the quaternion algebras $B$ that we need to consider have type number equal to one. It follows in these cases that the conjugacy class of $\Gamma_{S, \calO}$ does not depend on the chosen maximal order $\calO$ of $B$. When $t_B>1$ there will be $t_B$ conjugacy classes of maximal arithmetic groups arising from $B$ that must be considered.

The process described above yields a large number of maximal arithmetic groups. While it is possible that each of these groups contains a subgroup of covolume $16\pi^2$ or $32\pi^2$, not all of these groups may contain \emph{torsion-free} subgroups of these covolumes. Suppose for instance that $\Gamma_{S, \calO}$ contains elements of finite orders $m_1, \dots, m_t$. (Note that, because only finitely many cyclotomic extensions embed into the quaternion algebra $B$, there are only finitely many possible orders of an element of $\Gamma_{S,\calO}$ of finite order.) If $\Gamma$ is a finite index subgroup of $\Gamma_{S,\mathcal O}$ that is torsion-free then we necessarily have that
\begin{equation}
\lcm(m_1, \dots, m_t)\ {\mid}\ [\Gamma_{S, \calO} : \Gamma].
\end{equation}
We then use Lemma \ref{lem:embedtors} to rule out many possible $B$, $\calO$, and $S$, where conditions 2.\ and 3.\ are easily verified using \textsc{Magma}. This allows us to eliminate additional maximal arithmetic subgroups, and, when combined with Theorem \ref{theorem:degreebounds}, proves the following.

%%%%%%%%%%%%%%%%%%%%%%%%%
\begin{theorem}
Let $\Gamma \backslash \calH$ be an irreducible fake quadric. Then the stable subgroup $\Gamma_{st}$ of $\Gamma$ is (up to conjugacy) contained in one of the maximal arithmetic subgroups of $\Isom^+(\calH)$ enumerated in the Appendix.
\end{theorem}
%%%%%%%%%%%%%%%%%%%%%%%%%

%%%%%%%%%%%%%%%%%%%%%%%%%
\bibliographystyle{plain}
\bibliography{FakeQuadricI}
%%%%%%%%%%%%%%%%%%%%%%%%%

\newpage

%%%%%%%%%%%%%%%%%%%%%%%%%
\section*{Appendix}\label{section:appendix}
%%%%%%%%%%%%%%%%%%%%%%%%%

In this appendix we list maximal arithmetic subgroups of $\PGL_2(\R) \times \PGL_2(\R)$ that may contain the stable subgroup of the fundamental group of an irreducible fake quadric. Since the number of these commensurability classes is large, we employ some shortcuts in how we record our results, which we now describe.

We group maximal stable arithmetic subgroups first by the data
\[ n, d, D, N \]
where $n=[k:\Q]$ is the degree of the underlying totally real field $k$ and $d$ is its discriminant (for the commensurability classes below, the discriminant uniquely determines the field); the associated quaternion algebra $B$ is specified by the absolute norm $D=\N(\frakD)$ of its discriminant $\frakD$, and the Eichler order $\calE \subset B$ is specified by the absolute norm $N=\N(\frakN)$ of its level $\frakN$. There are only finitely many possibilities for $k,B,\calO$ with this data, and they are explicitly given in the computer readable output available online \cite{ourdata}.  

In each line of the tables, we provide a bit more data about the groups (which in some cases depends on more than just the data above, so the data $n,d,D,N$ may be repeated). First, we compute the covolume of the maximal stable arithmetic group with the specified data. Second, we compute the index of the maximal \emph{holomorphic} stable group inside the maximal stable group---in nearly all cases, this index is $4$ (coming from elements acting by orientation-reversing isometry on each of the two factors of $\calH$). Third, we compute \emph{a divisor} of the least common multiple of the orders of the elements of finite order in $\Gamma_{S,\mathcal O}$. 
% Third, we recall that the $\Aut(k)$-orbit of the ramification set $\Ram(B)$ determines the commensurability class \cite[Thm.\ 8.4.7]{Maclachlan--Reid}, so we record $\Aut(k)$. (Until the groups are explicitly computed, we cannot know if these Galois automorphisms will identify fake quadrics up to isomorphism.) 
Fourth, we compute the number $0 \le \fraks \le |S|$ that appears in \eqref{equation:volumeformulaequalityagain} using class field theory \cite[p.\ 356]{Maclachlan--Reid}. Finally, the last column records $\star$ if it is guaranteed that $\Gamma = \Gamma_{st}$; otherwise, we leave this entry blank, indicating that it is possible that there are unstable fake quadric groups $\Gamma$ with this data. 

%%%%%%%%%%%%%%%%%%%%%%%%%
\begin{rmk}\label{rmk:dzambicerror}
We now make some remarks on the discrepancies between our work and \Dzambic's \cite{Dzambic} on fake quadrics defined over quadratic fields. Any lattices in our paper that are not stable will not appear in \cite{Dzambic}. Unfortunately, our tables also differ for stable lattices. We found that the entry $[\Q(\sqrt{5}), v_2 v_{31}, \emptyset, 2]$ in \cite[Thm.\ 3.15]{Dzambic}, and similarly for $v_{31}^\prime$, cannot produce fake quadrics, even in the generality discussed in this paper.

We quickly explain how our index calculations \emph{should} differ from those of \cite{Dzambic} when $\frakN = 1$, i.e., for $\Gamma_\calO$. It is easy to see, using Eichler's theorem on norms, that the index of $\Gamma_\calO^+$ in $\Gamma_\calO$ is equal to $2$ when the narrow class number $h^+$ equals the class number $h$, since we can find an element of $\calO^*$ with reduced norm to $k$ that is negative at each real place of $k$, but cannot find one that is negative at exactly one real place. Otherwise, $h^+ = 2 h$ and $\Gamma_\calO^+$ has index $4$ in $\Gamma_\calO$, since we can find elements of $\calO^*$ with reduced norm of chosen sign at each real embedding. When there is the possibility that the lattice has proper stable subgroup, we report the index in $\Gamma_\calO$ of a (potential) lattice of covolume $32 \pi^2$. In this case, our index should be $4$ times the index $I$ from \cite{Dzambic} when $h^+ = h$ and $8$ times \Dzambic's $I$ when $h^+ = 2 h$. When $\Gamma_\calO$ is stable, we report the index of a (potential) subgroup of covolume $16 \pi^2$, so our index should be $2 I$ when $h^+ = h$ and $4 I$ when $h^+ = 2 h$. This, plus the index of $\Gamma_\calO^+$ in $N\Gamma_\calO^+$ (in the notation from \cite{Dzambic}), should account for the difference between our index and Dzambic's.
\end{rmk}
%%%%%%%%%%%%%%%%%%%%%%%%%

%%%%%%%%%%%%%%%%%%%%%%%%%
\newpage
%%%%%%%%%%%%%%%%%%%%%%%%%

%%%%%%%%%%%%%%%%%%%%%%%%%
\begin{footnotesize}

\begin{multicols}{2}
\begin{tabular}{ccc|ccccc}
$d$ & $D$ & $N$ & vol & hol & lcm & $\nu$ & st\\
\hline
5 & 20 & 1 & $(1/5)\pi^2$ & 4 & 5 & 0 &  \\
 &  & 9 & $\pi^2$ & 4 & 1 & 1 &  \\
 &  & 19 & $2\pi^2$ & 4 & 1 & 1 & $\star$ \\
 &  & 31 & $(16/5)\pi^2$ & 4 & 5 & 1 & $\star$ \\
 &  & 79 & $8\pi^2$ & 4 & 1 & 1 & $\star$ \\
 &  & 279 & $16\pi^2$ & 4 & 1 & 2 & $\star$ \\
 & 36 & 1 & $(2/5)\pi^2$ & 4 & 5 & 0 &  \\
 &  & 19 & $4\pi^2$ & 4 & 1 & 1 & $\star$ \\
 &  & 79 & $16\pi^2$ & 4 & 1 & 1 & $\star$ \\
 & 44 & 1 & $(1/2)\pi^2$ & 4 & 2 & 0 & $\star$ \\
 &  & 31 & $8\pi^2$ & 4 & 1 & 1 & $\star$ \\
 & 45 & 1 & $(8/15)\pi^2$ & 4 & 15 & 0 &  \\
 &  & 4 & $(4/3)\pi^2$ & 4 & 3 & 1 &  \\
 &  & 11 & $(16/5)\pi^2$ & 4 & 5 & 1 & $\star$ \\
 &  & 19 & $(16/3)\pi^2$ & 4 & 3 & 1 & $\star$ \\
 &  & 29 & $8\pi^2$ & 4 & 1 & 1 & $\star$ \\
 &  & 44 & $8\pi^2$ & 4 & 1 & 2 & $\star$ \\
 &  & 59 & $16\pi^2$ & 4 & 1 & 1 & $\star$ \\
 & 55 & 1 & $(2/3)\pi^2$ & 4 & 3 & 0 & $\star$ \\
 &  & 11 & $4\pi^2$ & 4 & 1 & 1 & $\star$ \\
 & 99 & 1 & $(4/3)\pi^2$ & 4 & 3 & 0 & $\star$ \\
 &  & 5 & $4\pi^2$ & 4 & 1 & 1 & $\star$ \\
 &  & 11 & $8\pi^2$ & 4 & 1 & 1 & $\star$ \\
 & 155 & 1 & $2\pi^2$ & 4 & 1 & 0 & $\star$ \\
 & 164 & 1 & $2\pi^2$ & 4 & 1 & 0 & $\star$ \\
 & 205 & 1 & $(8/3)\pi^2$ & 4 & 3 & 0 & $\star$ \\
 &  & 11 & $16\pi^2$ & 4 & 1 & 1 & $\star$ \\
 & 245 & 1 & $(16/5)\pi^2$ & 4 & 5 & 0 &  \\
 &  & 4 & $8\pi^2$ & 4 & 1 & 1 &  \\
 &  & 9 & $16\pi^2$ & 4 & 1 & 1 &  \\
 & 279 & 1 & $4\pi^2$ & 4 & 1 & 0 & $\star$ \\
 & 305 & 1 & $4\pi^2$ & 4 & 1 & 0 & $\star$ \\
 & 369 & 1 & $(16/3)\pi^2$ & 4 & 3 & 0 & $\star$ \\
 &  & 5 & $16\pi^2$ & 4 & 1 & 1 & $\star$ \\
 & 441 & 1 & $(32/5)\pi^2$ & 4 & 5 & 0 &  \\
 &  & 4 & $16\pi^2$ & 4 & 1 & 1 &  \\
 & 539 & 1 & $8\pi^2$ & 4 & 1 & 0 & $\star$ \\
 & 549 & 1 & $8\pi^2$ & 4 & 1 & 0 & $\star$ \\
 & 1205 & 1 & $16\pi^2$ & 4 & 1 & 0 & $\star$ \\
 & 1980 & 1 & $4\pi^2$ & 4 & 1 & 0 & $\star$ \\
 & 7380 & 1 & $16\pi^2$ & 4 & 1 & 0 & $\star$ \\
\end{tabular}

\begin{tabular}{ccc|ccccc}
$d$ & $D$ & $N$ & vol & hol & lcm & $\nu$ & st\\
\hline
8 & 14 & 1 & $(1/4)\pi^2$ & 4 & 4 & 0 & $\star$ \\
 &  & 7 & $\pi^2$ & 4 & 1 & 1 & $\star$ \\
 &  & 31 & $4\pi^2$ & 4 & 1 & 1 & $\star$ \\
 &  & 127 & $16\pi^2$ & 4 & 1 & 1 & $\star$ \\
 &  & 217 & $16\pi^2$ & 4 & 1 & 2 & $\star$ \\
 & 18 & 1 & $(1/3)\pi^2$ & 4 & 3 & 0 &  \\
 &  & 7 & $(4/3)\pi^2$ & 4 & 3 & 1 & $\star$ \\
 &  & 23 & $4\pi^2$ & 4 & 1 & 1 & $\star$ \\
 &  & 31 & $(16/3)\pi^2$ & 4 & 3 & 1 & $\star$ \\
 &  & 47 & $8\pi^2$ & 4 & 1 & 1 & $\star$ \\
 &  & 49 & $(16/3)\pi^2$ & 4 & 3 & 2 &  \\
 &  & 161 & $16\pi^2$ & 4 & 1 & 2 & $\star$ \\
 & 34 & 1 & $(2/3)\pi^2$ & 4 & 3 & 0 & $\star$ \\
 &  & 7 & $(8/3)\pi^2$ & 4 & 3 & 1 & $\star$ \\
 &  & 23 & $8\pi^2$ & 4 & 1 & 1 & $\star$ \\
 &  & 47 & $16\pi^2$ & 4 & 1 & 1 & $\star$ \\
 & 50 & 1 & $\pi^2$ & 4 & 1 & 0 &  \\
 &  & 7 & $4\pi^2$ & 4 & 1 & 1 & $\star$ \\
 &  & 31 & $16\pi^2$ & 4 & 1 & 1 & $\star$ \\
 &  & 49 & $16\pi^2$ & 4 & 1 & 2 &  \\
 & 63 & 1 & $2\pi^2$ & 4 & 1 & 0 & $\star$ \\
 &  & 7 & $8\pi^2$ & 4 & 1 & 1 & $\star$ \\
 & 119 & 1 & $4\pi^2$ & 4 & 1 & 0 & $\star$ \\
 &  & 7 & $16\pi^2$ & 4 & 1 & 1 & $\star$ \\
 & 153 & 1 & $(16/3)\pi^2$ & 4 & 3 & 0 & $\star$ \\
 &  & 2 & $8\pi^2$ & 4 & 1 & 1 & $\star$ \\
 & 194 & 1 & $4\pi^2$ & 4 & 1 & 0 & $\star$ \\
 &  & 7 & $16\pi^2$ & 4 & 1 & 1 & $\star$ \\
 & 225 & 1 & $8\pi^2$ & 4 & 1 & 0 &  \\
 & 289 & 1 & $(32/3)\pi^2$ & 4 & 3 & 0 &  \\
 &  & 2 & $16\pi^2$ & 4 & 1 & 1 &  \\
 & 386 & 1 & $8\pi^2$ & 4 & 1 & 0 & $\star$ \\
 & 425 & 1 & $16\pi^2$ & 4 & 1 & 0 & $\star$ \\
 & 2142 & 1 & $8\pi^2$ & 4 & 1 & 0 & $\star$ \\
 & 4046 & 1 & $16\pi^2$ & 4 & 1 & 0 & $\star$ \\
\end{tabular}

\end{multicols}
\begin{center}
\textbf{Table 1a}: Commensurability classes for degree $n=2$ (1 of 3)
\end{center}

\newpage

\begin{multicols}{2}
\begin{tabular}{ccc|ccccc}
$d$ & $D$ & $N$ & vol & hol & lcm & $\nu$ & st\\
\hline
12 & 6 & 1 & $(1/6)\pi^2$ & 4 & 6 & 0 &  \\
 &  & 11 & $\pi^2$ & 4 & 1 & 1 & $\star$ \\
 &  & 23 & $2\pi^2$ & 4 & 1 & 1 & $\star$ \\
 &  & 47 & $4\pi^2$ & 4 & 1 & 1 & $\star$ \\
 &  & 191 & $16\pi^2$ & 4 & 1 & 1 & $\star$ \\
 & 26 & 1 & $\pi^2$ & 4 & 1 & 0 & $\star$ \\
 &  & 3 & $2\pi^2$ & 4 & 1 & 1 & $\star$ \\
 & 39 & 1 & $2\pi^2$ & 4 & 1 & 0 & $\star$ \\
 & 50 & 1 & $2\pi^2$ & 4 & 1 & 0 &  \\
 &  & 3 & $4\pi^2$ & 4 & 1 & 1 &  \\
 & 75 & 1 & $4\pi^2$ & 4 & 1 & 0 &  \\
 & 98 & 1 & $4\pi^2$ & 4 & 1 & 0 &  \\
 &  & 3 & $8\pi^2$ & 4 & 1 & 1 &  \\
 & 147 & 1 & $8\pi^2$ & 4 & 1 & 0 &  \\
 & 194 & 1 & $8\pi^2$ & 4 & 1 & 0 & $\star$ \\
 &  & 3 & $16\pi^2$ & 4 & 1 & 1 & $\star$ \\
 & 291 & 1 & $16\pi^2$ & 4 & 1 & 0 & $\star$ \\
 & 386 & 1 & $16\pi^2$ & 4 & 1 & 0 & $\star$ \\
13 & 9 & 1 & $(1/3)\pi^2$ & 4 & 6 & 0 &  \\
 &  & 23 & $4\pi^2$ & 4 & 1 & 1 & $\star$ \\
 & 12 & 1 & $(1/2)\pi^2$ & 4 & 2 & 0 & $\star$ \\
 &  & 3 & $\pi^2$ & 4 & 1 & 1 & $\star$ \\
 & 39 & 1 & $2\pi^2$ & 4 & 1 & 0 & $\star$ \\
 &  & 3 & $4\pi^2$ & 4 & 1 & 1 & $\star$ \\
 & 51 & 1 & $(8/3)\pi^2$ & 4 & 3 & 0 & $\star$ \\
 &  & 3 & $(16/3)\pi^2$ & 4 & 3 & 1 & $\star$ \\
 & 68 & 1 & $4\pi^2$ & 4 & 1 & 0 & $\star$ \\
 &  & 3 & $8\pi^2$ & 4 & 1 & 1 & $\star$ \\
 &  & 9 & $16\pi^2$ & 4 & 1 & 2 & $\star$ \\
 & 75 & 1 & $4\pi^2$ & 4 & 1 & 0 & $\star$ \\
 &  & 3 & $8\pi^2$ & 4 & 1 & 1 & $\star$ \\
 & 147 & 1 & $8\pi^2$ & 4 & 1 & 0 & $\star$ \\
 &  & 3 & $16\pi^2$ & 4 & 1 & 1 & $\star$ \\
 & 221 & 1 & $16\pi^2$ & 4 & 1 & 0 & $\star$ \\
 & 612 & 1 & $4\pi^2$ & 4 & 1 & 0 & $\star$ \\
17 & 4 & 1 & $(1/6)\pi^2$ & 4 & 6 & 0 &  \\
 &  & 47 & $4\pi^2$ & 4 & 1 & 1 & $\star$ \\
 &  & 191 & $16\pi^2$ & 4 & 1 & 1 & $\star$ \\
 & 18 & 1 & $(4/3)\pi^2$ & 4 & 3 & 0 & $\star$ \\
 &  & 2 & $2\pi^2$ & 4 & 1 & 1 & $\star$ \\
 & 26 & 1 & $2\pi^2$ & 4 & 1 & 0 & $\star$ \\
\end{tabular}

\begin{tabular}{ccc|ccccc}
$d$ & $D$ & $N$ & vol & hol & lcm & $\nu$ & st\\
\hline
17 & 34 & 1 & $(8/3)\pi^2$ & 4 & 3 & 0 & $\star$ \\
 &  & 2 & $4\pi^2$ & 4 & 1 & 1 & $\star$ \\
 & 50 & 1 & $4\pi^2$ & 4 & 1 & 0 & $\star$ \\
 & 98 & 1 & $8\pi^2$ & 4 & 1 & 0 & $\star$ \\
 & 117 & 1 & $16\pi^2$ & 4 & 1 & 0 & $\star$ \\
 & 225 & 1 & $32\pi^2$ & 4 & 1 & 0 &  \\
 & 468 & 1 & $4\pi^2$ & 4 & 1 & 0 & $\star$ \\
 & 612 & 1 & $(16/3)\pi^2$ & 4 & 3 & 0 &  \\
 & 884 & 1 & $8\pi^2$ & 4 & 1 & 0 & $\star$ \\
 & 900 & 1 & $8\pi^2$ & 4 & 1 & 0 &  \\
 & 1700 & 1 & $16\pi^2$ & 4 & 1 & 0 &  \\
 & 1764 & 1 & $16\pi^2$ & 4 & 1 & 0 &  \\
 & 3332 & 1 & $32\pi^2$ & 4 & 1 & 0 &  \\
21 & 12 & 1 & $\pi^2$ & 4 & 2 & 0 &  \\
 &  & 7 & $4\pi^2$ & 4 & 1 & 1 &  \\
 & 15 & 1 & $(4/3)\pi^2$ & 4 & 3 & 0 & $\star$ \\
 &  & 5 & $4\pi^2$ & 4 & 1 & 1 & $\star$ \\
 &  & 7 & $(16/3)\pi^2$ & 4 & 3 & 1 & $\star$ \\
 &  & 35 & $16\pi^2$ & 4 & 1 & 2 & $\star$ \\
 & 20 & 1 & $2\pi^2$ & 4 & 1 & 0 & $\star$ \\
 &  & 3 & $4\pi^2$ & 4 & 1 & 1 & $\star$ \\
 &  & 7 & $8\pi^2$ & 4 & 1 & 1 & $\star$ \\
 &  & 21 & $16\pi^2$ & 4 & 1 & 2 & $\star$ \\
 & 21 & 1 & $2\pi^2$ & 4 & 2 & 0 &  \\
 & 25 & 1 & $(8/3)\pi^2$ & 4 & 3 & 0 &  \\
 &  & 3 & $(16/3)\pi^2$ & 4 & 3 & 1 &  \\
 &  & 7 & $(32/3)\pi^2$ & 4 & 3 & 1 &  \\
 & 35 & 1 & $4\pi^2$ & 4 & 1 & 0 & $\star$ \\
 &  & 3 & $8\pi^2$ & 4 & 1 & 1 & $\star$ \\
 & 51 & 1 & $(16/3)\pi^2$ & 4 & 3 & 0 & $\star$ \\
 &  & 5 & $16\pi^2$ & 4 & 1 & 1 & $\star$ \\
 & 68 & 1 & $8\pi^2$ & 4 & 1 & 0 & $\star$ \\
 &  & 3 & $16\pi^2$ & 4 & 1 & 1 & $\star$ \\
 & 119 & 1 & $16\pi^2$ & 4 & 1 & 0 & $\star$ \\
 & 300 & 1 & $4\pi^2$ & 4 & 1 & 0 &  \\
 &  & 7 & $16\pi^2$ & 4 & 1 & 1 &  \\
 & 525 & 1 & $8\pi^2$ & 4 & 1 & 0 &  \\
 & 1020 & 1 & $16\pi^2$ & 4 & 1 & 0 & $\star$ \\
\end{tabular}
\end{multicols}

\begin{center}
\textbf{Table 1b}: Commensurability classes for degree $n=2$ (2 of 3)
\end{center}
\newpage

\begin{multicols}{2}
\begin{tabular}{ccc|ccccc}
$d$ & $D$ & $N$ & vol & hol & lcm & $\nu$ & st\\
\hline
24 & 6 & 1 & $(1/2)\pi^2$ & 4 & 2 & 0 &  \\
 & 15 & 1 & $2\pi^2$ & 4 & 1 & 0 & $\star$ \\
 & 150 & 1 & $2\pi^2$ & 4 & 1 & 0 &  \\
28 & 6 & 1 & $(2/3)\pi^2$ & 4 & 3 & 0 & $\star$ \\
 &  & 3 & $(4/3)\pi^2$ & 4 & 3 & 1 & $\star$ \\
 &  & 7 & $(8/3)\pi^2$ & 4 & 3 & 1 & $\star$ \\
 &  & 21 & $(16/3)\pi^2$ & 4 & 3 & 2 & $\star$ \\
 &  & 47 & $16\pi^2$ & 4 & 1 & 1 & $\star$ \\
 & 9 & 1 & $(4/3)\pi^2$ & 4 & 6 & 0 &  \\
 &  & 2 & $2\pi^2$ & 4 & 2 & 1 &  \\
 &  & 7 & $(16/3)\pi^2$ & 4 & 3 & 1 &  \\
 &  & 14 & $8\pi^2$ & 4 & 1 & 2 &  \\
 & 14 & 1 & $2\pi^2$ & 4 & 1 & 0 &  \\
 &  & 3 & $4\pi^2$ & 4 & 1 & 1 & $\star$ \\
 &  & 9 & $8\pi^2$ & 4 & 1 & 2 &  \\
 & 21 & 1 & $4\pi^2$ & 4 & 2 & 0 & $\star$ \\
 &  & 3 & $8\pi^2$ & 4 & 1 & 1 & $\star$ \\
 & 50 & 1 & $8\pi^2$ & 2 & 1 & 0 &  \\
 &  & 3 & $16\pi^2$ & 4 & 1 & 1 & $\star$ \\
 &  & 7 & $32\pi^2$ & 4 & 1 & 1 &  \\
 &  & 9 & $32\pi^2$ & 4 & 1 & 2 &  \\
 & 75 & 1 & $16\pi^2$ & 4 & 1 & 0 & $\star$ \\
 & 126 & 1 & $2\pi^2$ & 4 & 1 & 0 &  \\
 & 450 & 1 & $8\pi^2$ & 4 & 1 & 0 &  \\
 &  & 7 & $32\pi^2$ & 4 & 1 & 1 &  \\
33 & 6 & 1 & $\pi^2$ & 4 & 2 & 0 & $\star$ \\
 &  & 31 & $16\pi^2$ & 4 & 1 & 1 & $\star$ \\
 & 51 & 1 & $16\pi^2$ & 4 & 1 & 0 & $\star$ \\
 & 204 & 1 & $4\pi^2$ & 4 & 1 & 0 & $\star$ \\
41 & 4 & 1 & $(2/3)\pi^2$ & 4 & 6 & 0 &  \\
 &  & 5 & $2\pi^2$ & 4 & 2 & 1 & $\star$ \\
 &  & 23 & $8\pi^2$ & 4 & 1 & 1 & $\star$ \\
 & 10 & 1 & $(8/3)\pi^2$ & 4 & 3 & 0 & $\star$ \\
 &  & 2 & $4\pi^2$ & 4 & 1 & 1 & $\star$ \\
 &  & 5 & $8\pi^2$ & 4 & 1 & 1 & $\star$ \\
 & 18 & 1 & $(16/3)\pi^2$ & 4 & 3 & 0 & $\star$ \\
 &  & 2 & $8\pi^2$ & 4 & 1 & 1 & $\star$ \\
 &  & 5 & $16\pi^2$ & 4 & 1 & 1 & $\star$ \\
 & 25 & 1 & $(32/3)\pi^2$ & 4 & 3 & 0 &  \\
 &  & 2 & $16\pi^2$ & 4 & 1 & 1 & $\star$ \\
 & 100 & 1 & $(8/3)\pi^2$ & 4 & 3 & 0 &  \\
 & 180 & 1 & $(16/3)\pi^2$ & 4 & 3 & 0 & $\star$ \\
 &  & 5 & $16\pi^2$ & 4 & 1 & 1 & $\star$ \\
\end{tabular}

\begin{tabular}{ccc|ccccc}
$d$ & $D$ & $N$ & vol & hol & lcm & $\nu$ & st\\
\hline
60 & 6 & 1 & $2\pi^2$ & 4 & 1 & 0 &  \\
 &  & 7 & $8\pi^2$ & 4 & 1 & 1 & $\star$ \\
 &  & 49 & $32\pi^2$ & 4 & 1 & 2 &  \\
 & 15 & 1 & $8\pi^2$ & 4 & 1 & 0 &  \\
65 & 4 & 1 & $(4/3)\pi^2$ & 4 & 6 & 0 &  \\
 &  & 5 & $4\pi^2$ & 4 & 2 & 1 &  \\
 &  & 7 & $(16/3)\pi^2$ & 4 & 3 & 1 & $\star$ \\
 &  & 35 & $16\pi^2$ & 4 & 1 & 2 & $\star$ \\
 & 10 & 1 & $(16/3)\pi^2$ & 4 & 3 & 0 & $\star$ \\
 &  & 2 & $8\pi^2$ & 4 & 1 & 1 & $\star$ \\
 & 14 & 1 & $8\pi^2$ & 4 & 2 & 0 & $\star$ \\
 & 18 & 2 & $16\pi^2$ & 4 & 1 & 1 & $\star$ \\
 & 26 & 1 & $16\pi^2$ & 4 & 1 & 0 & $\star$ \\
 & 140 & 1 & $8\pi^2$ & 4 & 1 & 0 & $\star$ \\
 & 180 & 1 & $(32/3)\pi^2$ & 4 & 3 & 0 &  \\
 & 252 & 1 & $16\pi^2$ & 4 & 1 & 0 & $\star$ \\
 & 260 & 1 & $16\pi^2$ & 4 & 1 & 0 &  \\
 & 468 & 1 & $32\pi^2$ & 4 & 1 & 0 &  \\
69 & 15 & 1 & $8\pi^2$ & 4 & 1 & 0 & $\star$ \\
145 & 4 & 1 & $(16/3)\pi^2$ & 4 & 6 & 0 &  \\
 &  & 5 & $16\pi^2$ & 4 & 2 & 1 &  \\
 & 36 & 1 & $(16/3)\pi^2$ & 4 & 6 & 0 &  \\
 &  & 5 & $16\pi^2$ & 4 & 2 & 1 &  \\
161 & 4 & 1 & $(16/3)\pi^2$ & 2 & 6 & 0 &  \\
\end{tabular}
\end{multicols}

\begin{center}
\textbf{Table 1c}: Commensurability classes for degree $n=2$ (3 of 3)
\end{center}
\newpage

\begin{multicols}{2}
\begin{tabular}{ccc|ccccc}
$d$ & $D$ & $N$ & vol & hol & lcm & $\nu$ & st\\
\hline
49 & 7 & 1 & $(1/7)\pi^2$ & 4 & 14 & 0 & $\star$ \\
 &  & 13 & $\pi^2$ & 4 & 2 & 1 & $\star$ \\
 &  & 27 & $2\pi^2$ & 4 & 1 & 1 & $\star$ \\
 &  & 223 & $16\pi^2$ & 4 & 1 & 1 & $\star$ \\
 & 8 & 1 & $(1/6)\pi^2$ & 4 & 6 & 0 & $\star$ \\
 &  & 7 & $(2/3)\pi^2$ & 4 & 3 & 1 & $\star$ \\
 & 13 & 1 & $(2/7)\pi^2$ & 4 & 7 & 0 & $\star$ \\
 &  & 7 & $(8/7)\pi^2$ & 4 & 7 & 1 & $\star$ \\
 &  & 13 & $2\pi^2$ & 4 & 1 & 1 & $\star$ \\
 &  & 27 & $4\pi^2$ & 4 & 1 & 1 & $\star$ \\
 &  & 91 & $8\pi^2$ & 4 & 1 & 2 & $\star$ \\
 &  & 189 & $16\pi^2$ & 4 & 1 & 2 & $\star$ \\
 & 29 & 1 & $(2/3)\pi^2$ & 4 & 3 & 0 & $\star$ \\
 &  & 7 & $(8/3)\pi^2$ & 4 & 3 & 1 & $\star$ \\
 & 43 & 1 & $\pi^2$ & 4 & 2 & 0 & $\star$ \\
 &  & 7 & $4\pi^2$ & 4 & 1 & 1 & $\star$ \\
 & 97 & 1 & $(16/7)\pi^2$ & 4 & 7 & 0 & $\star$ \\
 &  & 13 & $16\pi^2$ & 4 & 1 & 1 & $\star$ \\
 & 113 & 1 & $(8/3)\pi^2$ & 4 & 3 & 0 & $\star$ \\
 & 337 & 1 & $8\pi^2$ & 4 & 1 & 0 & $\star$ \\
 & 673 & 1 & $16\pi^2$ & 4 & 1 & 0 & $\star$ \\
81 & 3 & 1 & $(1/9)\pi^2$ & 4 & 18 & 0 & $\star$ \\
 &  & 8 & $(1/2)\pi^2$ & 4 & 2 & 1 & $\star$ \\
 &  & 17 & $\pi^2$ & 4 & 2 & 1 & $\star$ \\
 &  & 71 & $4\pi^2$ & 4 & 1 & 1 & $\star$ \\
 & 17 & 1 & $(8/9)\pi^2$ & 4 & 9 & 0 & $\star$ \\
 &  & 3 & $(16/9)\pi^2$ & 4 & 9 & 1 & $\star$ \\
 &  & 8 & $4\pi^2$ & 4 & 1 & 1 & $\star$ \\
 &  & 17 & $8\pi^2$ & 4 & 1 & 1 & $\star$ \\
 &  & 24 & $8\pi^2$ & 4 & 1 & 2 & $\star$ \\
 &  & 51 & $16\pi^2$ & 4 & 1 & 2 & $\star$ \\
 & 19 & 1 & $\pi^2$ & 4 & 2 & 0 & $\star$ \\
 &  & 3 & $2\pi^2$ & 4 & 1 & 1 & $\star$ \\
 & 37 & 1 & $2\pi^2$ & 4 & 1 & 0 & $\star$ \\
 &  & 3 & $4\pi^2$ & 4 & 1 & 1 & $\star$ \\
 & 73 & 1 & $4\pi^2$ & 4 & 1 & 0 & $\star$ \\
 &  & 3 & $8\pi^2$ & 4 & 1 & 1 & $\star$ \\
 & 969 & 1 & $8\pi^2$ & 4 & 1 & 0 & $\star$ \\
 & 1887 & 1 & $16\pi^2$ & 4 & 1 & 0 & $\star$ \\
\end{tabular}

\begin{tabular}{ccc|ccccc}
$d$ & $D$ & $N$ & vol & hol & lcm & $\nu$ & st\\
\hline
148 & 2 & 1 & $(1/6)\pi^2$ & 4 & 6 & 0 & $\star$ \\
 &  & 5 & $(1/2)\pi^2$ & 4 & 2 & 1 & $\star$ \\
 &  & 23 & $2\pi^2$ & 4 & 1 & 1 & $\star$ \\
 &  & 31 & $(8/3)\pi^2$ & 4 & 3 & 1 & $\star$ \\
 &  & 155 & $8\pi^2$ & 4 & 1 & 2 & $\star$ \\
 &  & 191 & $16\pi^2$ & 4 & 1 & 1 & $\star$ \\
 & 5 & 1 & $(2/3)\pi^2$ & 4 & 3 & 0 & $\star$ \\
 &  & 2 & $\pi^2$ & 4 & 1 & 1 & $\star$ \\
 &  & 23 & $8\pi^2$ & 4 & 1 & 1 & $\star$ \\
 &  & 62 & $16\pi^2$ & 4 & 1 & 2 & $\star$ \\
 & 13 & 1 & $2\pi^2$ & 4 & 1 & 0 & $\star$ \\
 & 17 & 1 & $(8/3)\pi^2$ & 4 & 3 & 0 & $\star$ \\
 &  & 2 & $4\pi^2$ & 4 & 1 & 1 & $\star$ \\
 &  & 5 & $8\pi^2$ & 4 & 1 & 1 & $\star$ \\
 & 25 & 1 & $4\pi^2$ & 4 & 1 & 0 & $\star$ \\
 & 97 & 1 & $16\pi^2$ & 4 & 1 & 0 & $\star$ \\
 & 130 & 1 & $2\pi^2$ & 4 & 1 & 0 & $\star$ \\
 & 170 & 1 & $(8/3)\pi^2$ & 4 & 3 & 0 & $\star$ \\
 & 250 & 1 & $4\pi^2$ & 4 & 1 & 0 & $\star$ \\
 & 442 & 1 & $8\pi^2$ & 4 & 1 & 0 & $\star$ \\
 & 850 & 1 & $16\pi^2$ & 4 & 1 & 0 & $\star$ \\
 & 970 & 1 & $16\pi^2$ & 4 & 1 & 0 & $\star$ \\
169 & 5 & 1 & $(2/3)\pi^2$ & 4 & 3 & 0 & $\star$ \\
 &  & 5 & $2\pi^2$ & 4 & 1 & 1 & $\star$ \\
 &  & 47 & $16\pi^2$ & 4 & 1 & 1 & $\star$ \\
 & 13 & 1 & $2\pi^2$ & 4 & 1 & 0 & $\star$ \\
 & 125 & 1 & $(8/3)\pi^2$ & 4 & 3 & 0 & $\star$ \\
 & 325 & 1 & $8\pi^2$ & 4 & 1 & 0 & $\star$ \\
229 & 2 & 1 & $(1/3)\pi^2$ & 4 & 6 & 0 & $\star$ \\
 &  & 7 & $(4/3)\pi^2$ & 4 & 3 & 1 & $\star$ \\
 &  & 23 & $4\pi^2$ & 4 & 1 & 1 & $\star$ \\
 &  & 31 & $(16/3)\pi^2$ & 4 & 3 & 1 & $\star$ \\
 &  & 47 & $8\pi^2$ & 4 & 1 & 1 & $\star$ \\
 &  & 161 & $16\pi^2$ & 4 & 1 & 2 & $\star$ \\
 & 4 & 1 & $\pi^2$ & 4 & 2 & 0 & $\star$ \\
 &  & 7 & $4\pi^2$ & 4 & 1 & 1 & $\star$ \\
 &  & 31 & $16\pi^2$ & 4 & 1 & 1 & $\star$ \\
 & 7 & 1 & $2\pi^2$ & 4 & 2 & 0 & $\star$ \\
 & 13 & 1 & $4\pi^2$ & 4 & 1 & 0 & $\star$ \\
 &  & 7 & $16\pi^2$ & 4 & 1 & 1 & $\star$ \\
 & 49 & 1 & $16\pi^2$ & 4 & 1 & 0 & $\star$ \\
\end{tabular}
\end{multicols}
\begin{center}
\textbf{Table 2a}: Commensurability classes for degree $n=3$ (1 of 2)
\end{center}
\newpage

\begin{multicols}{2}
\begin{tabular}{ccc|ccccc}
$d$ & $D$ & $N$ & vol & hol & lcm & $\nu$ & st\\
\hline
257 & 3 & 1 & $(2/3)\pi^2$ & 4 & 6 & 0 & $\star$ \\
 &  & 5 & $2\pi^2$ & 4 & 2 & 1 & $\star$ \\
 &  & 7 & $(8/3)\pi^2$ & 4 & 3 & 1 & $\star$ \\
 &  & 35 & $8\pi^2$ & 4 & 1 & 2 & $\star$ \\
 &  & 47 & $16\pi^2$ & 4 & 1 & 1 & $\star$ \\
 & 5 & 1 & $(4/3)\pi^2$ & 4 & 3 & 0 & $\star$ \\
 &  & 3 & $(8/3)\pi^2$ & 4 & 3 & 1 & $\star$ \\
 &  & 7 & $(16/3)\pi^2$ & 4 & 3 & 1 & $\star$ \\
 & 7 & 1 & $2\pi^2$ & 4 & 2 & 0 & $\star$ \\
 &  & 3 & $4\pi^2$ & 4 & 1 & 1 & $\star$ \\
 & 9 & 1 & $(8/3)\pi^2$ & 4 & 3 & 0 & $\star$ \\
 &  & 3 & $(16/3)\pi^2$ & 4 & 3 & 1 & $\star$ \\
 &  & 5 & $8\pi^2$ & 4 & 1 & 1 & $\star$ \\
 &  & 15 & $16\pi^2$ & 4 & 1 & 2 & $\star$ \\
 & 25 & 1 & $8\pi^2$ & 4 & 1 & 0 & $\star$ \\
 &  & 3 & $16\pi^2$ & 4 & 1 & 1 & $\star$ \\
 & 49 & 1 & $16\pi^2$ & 4 & 1 & 0 & $\star$ \\
 & 105 & 1 & $4\pi^2$ & 4 & 1 & 0 & $\star$ \\
 & 135 & 1 & $(16/3)\pi^2$ & 4 & 3 & 0 & $\star$ \\
 & 189 & 1 & $8\pi^2$ & 4 & 1 & 0 & $\star$ \\
 & 315 & 1 & $16\pi^2$ & 4 & 1 & 0 & $\star$ \\
 & 375 & 1 & $16\pi^2$ & 4 & 1 & 0 & $\star$ \\
316 & 2 & 1 & $(2/3)\pi^2$ & 4 & 3 & 0 & $\star$ \\
 &  & 1 & $(2/3)\pi^2$ & 4 & 6 & 0 & $\star$ \\
 &  & 2 & $\pi^2$ & 4 & 1 & 1 & $\star$ \\
 &  & 2 & $\pi^2$ & 4 & 2 & 1 & $\star$ \\
 &  & 11 & $4\pi^2$ & 4 & 1 & 1 & $\star$ \\
 &  & 23 & $8\pi^2$ & 4 & 1 & 1 & $\star$ \\
 &  & 62 & $16\pi^2$ & 4 & 1 & 2 & $\star$ \\
 & 17 & 2 & $16\pi^2$ & 4 & 1 & 1 & $\star$ \\
 & 68 & 1 & $(8/3)\pi^2$ & 4 & 3 & 0 & $\star$ \\
 &  & 11 & $16\pi^2$ & 4 & 1 & 1 & $\star$ \\
\end{tabular}

\begin{tabular}{ccc|ccccc}
$d$ & $D$ & $N$ & vol & hol & lcm & $\nu$ & st \\
\hline
321 & 3 & 1 & $\pi^2$ & 4 & 2 & 0 & $\star$ \\
 &  & 3 & $2\pi^2$ & 4 & 1 & 1 & $\star$ \\
 &  & 7 & $4\pi^2$ & 4 & 1 & 1 & $\star$ \\
 &  & 21 & $8\pi^2$ & 4 & 1 & 2 & $\star$ \\
 &  & 31 & $16\pi^2$ & 4 & 1 & 1 & $\star$ \\
697 & 5 & 1 & $(16/3)\pi^2$ & 4 & 3 & 0 & $\star$ \\
 & 13 & 1 & $16\pi^2$ & 4 & 1 & 0 & $\star$ \\
837 & 2 & 1 & $(8/3)\pi^2$ & 4 & 6 & 0 & $\star$ \\
 &  & 3 & $(16/3)\pi^2$ & 4 & 3 & 1 & $\star$ \\
 &  & 5 & $8\pi^2$ & 4 & 2 & 1 & $\star$ \\
 &  & 15 & $16\pi^2$ & 4 & 1 & 2 & $\star$ \\
 & 3 & 2 & $8\pi^2$ & 4 & 2 & 1 & $\star$ \\
 & 4 & 1 & $8\pi^2$ & 4 & 2 & 0 & $\star$ \\
 &  & 3 & $16\pi^2$ & 4 & 1 & 1 & $\star$ \\
 & 5 & 2 & $16\pi^2$ & 4 & 1 & 1 & $\star$ \\
 & 24 & 1 & $4\pi^2$ & 4 & 2 & 0 & $\star$ \\
 & 30 & 1 & $(16/3)\pi^2$ & 4 & 3 & 0 & $\star$ \\
 & 40 & 1 & $8\pi^2$ & 4 & 1 & 0 & $\star$ \\
 &  & 3 & $16\pi^2$ & 4 & 1 & 1 & $\star$ \\
 & 60 & 1 & $16\pi^2$ & 4 & 1 & 0 & $\star$ \\
 & 78 & 1 & $16\pi^2$ & 4 & 1 & 0 & $\star$ \\
1257 & 3 & 1 & $8\pi^2$ & 4 & 2 & 0 & $\star$ \\
 &  & 3 & $16\pi^2$ & 4 & 1 & 1 & $\star$ \\
\end{tabular}
\end{multicols}
\begin{center}
\textbf{Table 2b}: Commensurability classes for degree $n=3$ (2 of 2)
\end{center}
\newpage

\begin{multicols}{2}
\begin{tabular}{ccc|ccccc}
$d$ & $D$ & $N$ & vol & hol & lcm & $\nu$ & st\\
\hline
725 & 1 & 1 & $(1/15)\pi^2$ & 4 & 30 & 0 & $\star$ \\
 &  & 1 & $(1/15)\pi^2$ & 4 & 30 & 0 &  \\
 &  & 11 & $(2/5)\pi^2$ & 4 & 5 & 1 & $\star$ \\
 &  & 19 & $(2/3)\pi^2$ & 4 & 3 & 1 & $\star$ \\
 &  & 29 & $\pi^2$ & 4 & 2 & 1 & $\star$ \\
 &  & 29 & $\pi^2$ & 4 & 2 & 1 &  \\
 &  & 31 & $(16/15)\pi^2$ & 4 & 15 & 1 & $\star$ \\
 &  & 79 & $(8/3)\pi^2$ & 4 & 3 & 1 & $\star$ \\
 &  & 209 & $4\pi^2$ & 4 & 1 & 2 & $\star$ \\
 &  & 479 & $16\pi^2$ & 4 & 1 & 1 & $\star$ \\
 &  & 869 & $16\pi^2$ & 4 & 1 & 2 & $\star$ \\
 &  & 899 & $16\pi^2$ & 4 & 1 & 2 & $\star$ \\
 & 275 & 1 & $4\pi^2$ & 4 & 1 & 0 & $\star$ \\
 & 539 & 1 & $8\pi^2$ & 4 & 1 & 0 & $\star$ \\
 & 1025 & 1 & $16\pi^2$ & 4 & 1 & 0 & $\star$ \\
 & 2025 & 1 & $32\pi^2$ & 4 & 1 & 0 &  \\
1125 & 1 & 1 & $(2/15)\pi^2$ & 2 & 30 & 0 & $\star$ \\
 &  & 1 & $(2/15)\pi^2$ & 2 & 30 & 0 &  \\
 &  & 5 & $(2/5)\pi^2$ & 4 & 10 & 1 & $\star$ \\
 &  & 5 & $(2/5)\pi^2$ & 4 & 10 & 1 &  \\
 &  & 9 & $(2/3)\pi^2$ & 4 & 6 & 1 & $\star$ \\
 &  & 9 & $(2/3)\pi^2$ & 4 & 6 & 1 &  \\
 &  & 29 & $2\pi^2$ & 4 & 2 & 1 & $\star$ \\
 &  & 45 & $2\pi^2$ & 4 & 2 & 2 & $\star$ \\
 &  & 45 & $2\pi^2$ & 4 & 2 & 2 &  \\
 &  & 59 & $4\pi^2$ & 4 & 1 & 1 & $\star$ \\
 &  & 239 & $16\pi^2$ & 4 & 1 & 1 & $\star$ \\
 & 45 & 1 & $(16/15)\pi^2$ & 4 & 15 & 0 & $\star$ \\
 &  & 1 & $(16/15)\pi^2$ & 4 & 15 & 0 &  \\
 &  & 29 & $16\pi^2$ & 4 & 1 & 1 & $\star$ \\
 & 80 & 1 & $2\pi^2$ & 4 & 1 & 0 & $\star$ \\
 &  & 1 & $2\pi^2$ & 4 & 1 & 0 &  \\
 & 144 & 1 & $4\pi^2$ & 4 & 1 & 0 & $\star$ \\
 &  & 1 & $4\pi^2$ & 4 & 1 & 0 &  \\
 & 155 & 1 & $4\pi^2$ & 4 & 1 & 0 & $\star$ \\
 & 279 & 1 & $8\pi^2$ & 4 & 1 & 0 & $\star$ \\
 & 305 & 1 & $8\pi^2$ & 4 & 1 & 0 & $\star$ \\
 & 549 & 1 & $16\pi^2$ & 4 & 1 & 0 & $\star$ \\
 & 605 & 1 & $16\pi^2$ & 4 & 1 & 0 & $\star$ \\
 &  & 1 & $16\pi^2$ & 4 & 1 & 0 &  \\
 & 1089 & 1 & $32\pi^2$ & 4 & 1 & 0 &  \\
\end{tabular}

\begin{tabular}{ccc|ccccc}
$d$ & $D$ & $N$ & vol & hol & lcm & $\nu$ & st \\
\hline
1957 & 1 & 1 & $(1/3)\pi^2$ & 4 & 6 & 0 & $\star$ \\
 &  & 3 & $(2/3)\pi^2$ & 4 & 3 & 1 & $\star$ \\
 &  & 7 & $(4/3)\pi^2$ & 4 & 3 & 1 & $\star$ \\
 &  & 21 & $(8/3)\pi^2$ & 4 & 3 & 2 & $\star$ \\
 &  & 23 & $4\pi^2$ & 4 & 1 & 1 & $\star$ \\
 &  & 31 & $(16/3)\pi^2$ & 4 & 3 & 1 & $\star$ \\
 &  & 47 & $8\pi^2$ & 4 & 1 & 1 & $\star$ \\
 &  & 69 & $8\pi^2$ & 4 & 1 & 2 & $\star$ \\
 &  & 141 & $16\pi^2$ & 4 & 1 & 2 & $\star$ \\
 &  & 161 & $16\pi^2$ & 4 & 1 & 2 & $\star$ \\
 & 21 & 1 & $\pi^2$ & 4 & 2 & 0 & $\star$ \\
 &  & 31 & $16\pi^2$ & 4 & 1 & 1 & $\star$ \\
 & 291 & 1 & $16\pi^2$ & 4 & 1 & 0 & $\star$ \\
2000 & 20 & 1 & $\pi^2$ & 4 & 1 & 0 & $\star$ \\
 &  & 1 & $\pi^2$ & 4 & 1 & 0 &  \\
2304 & 18 & 1 & $\pi^2$ & 4 & 1 & 0 &  \\
2777 & 1 & 1 & $(2/3)\pi^2$ & 4 & 6 & 0 & $\star$ \\
 &  & 2 & $\pi^2$ & 4 & 2 & 1 & $\star$ \\
 &  & 11 & $4\pi^2$ & 4 & 1 & 1 & $\star$ \\
 &  & 23 & $8\pi^2$ & 4 & 1 & 1 & $\star$ \\
 &  & 47 & $16\pi^2$ & 4 & 1 & 1 & $\star$ \\
 &  & 62 & $16\pi^2$ & 4 & 1 & 2 & $\star$ \\
 & 194 & 1 & $16\pi^2$ & 4 & 1 & 0 & $\star$ \\
3600 & 99 & 1 & $16\pi^2$ & 4 & 1 & 0 & $\star$ \\
3981 & 15 & 1 & $2\pi^2$ & 4 & 1 & 0 & $\star$ \\
 & 27 & 1 & $4\pi^2$ & 4 & 1 & 0 & $\star$ \\
4225 & 1 & 1 & $(16/15)\pi^2$ & 4 & 30 & 0 &  \\
 &  & 4 & $(8/3)\pi^2$ & 4 & 6 & 1 & $\star$ \\
 &  & 4 & $(8/3)\pi^2$ & 4 & 6 & 1 &  \\
 &  & 9 & $(16/3)\pi^2$ & 4 & 6 & 1 &  \\
 & 36 & 4 & $16\pi^2$ & 4 & 1 & 1 & $\star$ \\
4352 & 1 & 1 & $(4/3)\pi^2$ & 2 & 12 & 0 & $\star$ \\
 &  & 1 & $(4/3)\pi^2$ & 2 & 12 & 0 &  \\
 &  & 2 & $2\pi^2$ & 2 & 4 & 1 & $\star$ \\
 &  & 2 & $2\pi^2$ & 2 & 4 & 1 &  \\
 &  & 7 & $(16/3)\pi^2$ & 4 & 3 & 1 & $\star$ \\
 &  & 14 & $8\pi^2$ & 4 & 1 & 2 & $\star$ \\
 &  & 23 & $16\pi^2$ & 4 & 1 & 1 & $\star$ \\
 &  & 98 & $32\pi^2$ & 4 & 1 & 3 &  \\
 & 14 & 1 & $2\pi^2$ & 4 & 1 & 0 & $\star$ \\
 &  & 7 & $8\pi^2$ & 4 & 1 & 1 & $\star$ \\
 & 34 & 1 & $(16/3)\pi^2$ & 2 & 3 & 0 & $\star$ \\
 &  & 1 & $(16/3)\pi^2$ & 2 & 3 & 0 &  \\
 & 98 & 1 & $16\pi^2$ & 2 & 1 & 0 & $\star$ \\
 &  & 1 & $16\pi^2$ & 2 & 1 & 0 &  \\
\end{tabular}
\end{multicols}
\begin{center}
\textbf{Table 3a}: Commensurability classes for degree $n=4$ (1 of 2)
\end{center}

\newpage

\begin{multicols}{2}
\begin{tabular}{ccc|ccccc}
$d$ & $D$ & $N$ & vol & hol & lcm & $\nu$ & st\\
\hline
4752 & 1 & 1 & $(4/3)\pi^2$ & 2 & 6 & 0 & $\star$ \\
 &  & 1 & $(4/3)\pi^2$ & 2 & 6 & 0 &  \\
 &  & 3 & $(8/3)\pi^2$ & 4 & 3 & 1 & $\star$ \\
 &  & 3 & $(8/3)\pi^2$ & 4 & 3 & 1 &  \\
 &  & 11 & $8\pi^2$ & 4 & 1 & 1 & $\star$ \\
 &  & 11 & $8\pi^2$ & 4 & 1 & 1 &  \\
 &  & 23 & $16\pi^2$ & 4 & 1 & 1 & $\star$ \\
 &  & 33 & $16\pi^2$ & 4 & 1 & 2 & $\star$ \\
 &  & 33 & $16\pi^2$ & 4 & 1 & 2 &  \\
 & 12 & 1 & $2\pi^2$ & 4 & 1 & 0 & $\star$ \\
 &  & 1 & $2\pi^2$ & 4 & 1 & 0 &  \\
 & 39 & 1 & $8\pi^2$ & 4 & 1 & 0 & $\star$ \\
4913 & 1 & 1 & $(4/3)\pi^2$ & 4 & 6 & 0 & $\star$ \\
 &  & 1 & $(4/3)\pi^2$ & 4 & 6 & 0 &  \\
 & 68 & 1 & $16\pi^2$ & 4 & 1 & 0 & $\star$ \\
 &  & 1 & $16\pi^2$ & 4 & 1 & 0 &  \\
5125 & 45 & 1 & $(32/3)\pi^2$ & 4 & 3 & 0 & $\star$ \\
6809 & 1 & 1 & $(8/3)\pi^2$ & 4 & 6 & 0 & $\star$ \\
 &  & 2 & $4\pi^2$ & 4 & 1 & 1 & $\star$ \\
 &  & 5 & $8\pi^2$ & 4 & 1 & 1 & $\star$ \\
 &  & 11 & $16\pi^2$ & 4 & 1 & 1 & $\star$ \\
 & 10 & 1 & $(8/3)\pi^2$ & 4 & 3 & 0 & $\star$ \\
 &  & 11 & $16\pi^2$ & 4 & 1 & 1 & $\star$ \\
7056 & 1 & 1 & $(8/3)\pi^2$ & 2 & 6 & 0 &  \\
 &  & 3 & $(16/3)\pi^2$ & 4 & 3 & 1 & $\star$ \\
 &  & 3 & $(16/3)\pi^2$ & 4 & 3 & 1 &  \\
 &  & 9 & $(32/3)\pi^2$ & 4 & 3 & 2 &  \\
 & 9 & 1 & $(8/3)\pi^2$ & 4 & 6 & 0 &  \\
 & 12 & 1 & $4\pi^2$ & 4 & 1 & 0 & $\star$ \\
 &  & 1 & $4\pi^2$ & 4 & 1 & 0 &  \\
 &  & 3 & $8\pi^2$ & 4 & 1 & 1 & $\star$ \\
 &  & 3 & $8\pi^2$ & 4 & 1 & 1 &  \\
 & 75 & 1 & $32\pi^2$ & 4 & 1 & 0 & $\star$ \\
7625 & 20 & 1 & $8\pi^2$ & 4 & 1 & 0 & $\star$ \\
\end{tabular}

\begin{tabular}{ccc|ccccc}
$d$ & $D$ & $N$ & vol & hol & lcm & $\nu$ & st \\
\hline
8069 & 1 & 1 & $(4/3)\pi^2$ & 4 & 6 & 0 & $\star$ \\
 &  & 1 & $(8/3)\pi^2$ & 2 & 6 & 0 & $\star$ \\
 &  & 5 & $8\pi^2$ & 4 & 2 & 0 & $\star$ \\
 &  & 5 & $8\pi^2$ & 4 & 2 & 1 & $\star$ \\
 & 35 & 1 & $16\pi^2$ & 4 & 1 & 0 & $\star$ \\
9248 & 1 & 1 & $(8/3)\pi^2$ & 4 & 6 & 0 &  \\
 &  & 1 & $(16/3)\pi^2$ & 2 & 6 & 0 &  \\
 &  & 2 & $8\pi^2$ & 4 & 2 & 0 &  \\
 &  & 2 & $8\pi^2$ & 4 & 2 & 1 & $\star$ \\
 &  & 2 & $8\pi^2$ & 4 & 2 & 1 &  \\
 & 4 & 1 & $(4/3)\pi^2$ & 4 & 3 & 0 & $\star$ \\
 &  & 1 & $(4/3)\pi^2$ & 4 & 3 & 0 &  \\
 & 26 & 1 & $16\pi^2$ & 4 & 1 & 0 & $\star$ \\
9909 & 15 & 1 & $8\pi^2$ & 4 & 1 & 0 & $\star$ \\
10273 & 1 & 2 & $8\pi^2$ & 4 & 2 & 1 & $\star$ \\
 &  & 6 & $16\pi^2$ & 4 & 1 & 2 & $\star$ \\
 & 6 & 1 & $(8/3)\pi^2$ & 4 & 6 & 0 & $\star$ \\
 & 26 & 1 & $16\pi^2$ & 4 & 1 & 0 & $\star$ \\
10889 & 1 & 1 & $(8/3)\pi^2$ & 4 & 6 & 0 & $\star$ \\
 &  & 2 & $8\pi^2$ & 4 & 2 & 0 & $\star$ \\
 &  & 2 & $8\pi^2$ & 4 & 2 & 1 & $\star$ \\
 &  & 11 & $16\pi^2$ & 4 & 1 & 1 & $\star$ \\
 & 14 & 1 & $8\pi^2$ & 4 & 2 & 0 & $\star$ \\
13068 & 6 & 1 & $4\pi^2$ & 4 & 1 & 0 & $\star$ \\
 &  & 1 & $4\pi^2$ & 4 & 1 & 0 &  \\
\end{tabular}
\end{multicols}
\begin{center}
\textbf{Table 3b}: Commensurability classes for degree $n=4$ (2 of 2)
\end{center}
\newpage

\begin{center}
\begin{tabular}{ccc|ccccc}
$d$ & $D$ & $N$ & vol & hol & lcm & $\nu$ & st\\
\hline
24217 & 5 & 1 & $(2/3)\pi^2$ & 4 & 3 & 0 & $\star$ \\
 &  & 23 & $8\pi^2$ & 4 & 1 & 1 & $\star$ \\
 &  & 47 & $16\pi^2$ & 4 & 1 & 1 & $\star$ \\
 & 17 & 1 & $(8/3)\pi^2$ & 4 & 3 & 0 & $\star$ \\
 &  & 5 & $8\pi^2$ & 4 & 1 & 1 & $\star$ \\
 & 97 & 1 & $16\pi^2$ & 4 & 1 & 0 & $\star$ \\
36497 & 3 & 1 & $(2/3)\pi^2$ & 4 & 6 & 0 & $\star$ \\
 &  & 23 & $8\pi^2$ & 4 & 1 & 1 & $\star$ \\
 &  & 47 & $16\pi^2$ & 4 & 1 & 1 & $\star$ \\
 & 13 & 1 & $4\pi^2$ & 4 & 1 & 0 & $\star$ \\
 &  & 3 & $8\pi^2$ & 4 & 1 & 1 & $\star$ \\
 & 25 & 1 & $8\pi^2$ & 4 & 1 & 0 & $\star$ \\
 &  & 3 & $16\pi^2$ & 4 & 1 & 1 & $\star$ \\
 & 49 & 1 & $16\pi^2$ & 4 & 1 & 0 & $\star$ \\
38569 & 7 & 1 & $2\pi^2$ & 4 & 2 & 0 & $\star$ \\
 & 13 & 1 & $4\pi^2$ & 4 & 1 & 0 & $\star$ \\
 &  & 7 & $16\pi^2$ & 4 & 1 & 1 & $\star$ \\
 & 17 & 1 & $(16/3)\pi^2$ & 4 & 3 & 0 & $\star$ \\
81509 & 2 & 1 & $(4/3)\pi^2$ & 4 & 6 & 0 & $\star$ \\
 & 9 & 2 & $16\pi^2$ & 4 & 1 & 1 & $\star$ \\
81589 & 2 & 1 & $(4/3)\pi^2$ & 4 & 6 & 0 & $\star$ \\
 &  & 11 & $8\pi^2$ & 4 & 1 & 1 & $\star$ \\
 & 13 & 1 & $16\pi^2$ & 4 & 1 & 0 & $\star$ \\
89417 & 3 & 1 & $(8/3)\pi^2$ & 4 & 6 & 0 & $\star$ \\
 &  & 5 & $8\pi^2$ & 4 & 2 & 1 & $\star$ \\
 &  & 11 & $16\pi^2$ & 4 & 1 & 1 & $\star$ \\
 & 5 & 1 & $(16/3)\pi^2$ & 4 & 3 & 0 & $\star$ \\
138917 & 4 & 1 & $8\pi^2$ & 4 & 2 & 0 & $\star$ \\
 &  & 3 & $16\pi^2$ & 4 & 1 & 1 & $\star$ \\
\end{tabular}

\end{center}
\begin{center}
\textbf{Table 4}: Commensurability classes for degree $n=5$ (1 of 1)
\end{center}

\vspace{3ex}

\begin{center}
\begin{tabular}{ccc|ccccc}
$d$ & $D$ & $N$ & vol & hol & lcm & $\nu$ & st\\
\hline
1134389 & 1 & 1 & $(4/3)\pi^2$ & 4 & 6 & 0 & $\star$ \\
 &  & 1 & $(8/3)\pi^2$ & 2 & 6 & 0 & $\star$ \\
\end{tabular}
\end{center}
\begin{center}
\textbf{Table 5}: Commensurability classes for degree $n=6$ (1 of 1)
\end{center}

\end{footnotesize}

\end{document}